\documentclass[a4paper,11pt]{article}
\usepackage[utf8]{inputenc}
\usepackage{mathtools}
\usepackage[hypertexnames=false]{hyperref}
\usepackage{accents,amssymb,amsfonts,amsthm,bbm}
\usepackage{fullpage}
\usepackage{color}
\usepackage{tikz}

\newtheorem{theorem}{Theorem}[section]
\newtheorem{lemma}[theorem]{Lemma}
\newtheorem{conjecture}[theorem]{Conjecture}
\newtheorem{proposition}[theorem]{Proposition}
\newtheorem{definition}[theorem]{Definition}
\newtheorem{rem}[theorem]{Remark}

\newcommand{\N}{\mathbb N}
\newcommand{\Z}{\mathbb Z}
\newcommand{\R}{\mathbb R}

\renewcommand{\P}{\mathbb P}

\newcommand{\dd}{\text{\upshape{d}}}
\newcommand{\relmiddle}[1]{\mathrel{}\middle#1\mathrel{}}
\newcommand{\cond}{\relmiddle{|}}
\renewcommand{\epsilon}{\varepsilon}

\newcommand{\Pois}{\text{Pois}}

\newcommand{\cadlag}{c\`{a}dl\`{a}g }
\DeclareMathOperator*{\osc}{osc}
\newcommand{\Pk}{{P_k}}
\newcommand{\Ek}{{E_k}}

\newcommand\ex{\operatorname{ex}}
\newcommand{\at}[1]{a_{\ex(#1)}^-}
\newcommand{\bt}[1]{b_{\ex(#1)}^-}
\newcommand{\gt}[1]{g_{\ex(#1)}^-}
\newcommand{\gtt}[1]{g_{\ex(#1)-1}^-}
\newcommand{\betat}[1]{\beta^-_{\ex(#1)}}
\newcommand{\alphat}[1]{\alpha^-_{\ex(#1)}}

\newcommand{\A}{{4}}
\newcommand{\B}{{16}}

\begin{document}

\title{Extremal regime for one-dimensional\\Mott variable-range hopping}
\author{David A. Croydon\footnote{\tiny{Research Institute for Mathematical Sciences, Kyoto University, Kyoto 606--8502, Japan; croydon@kurims.kyoto-u.ac.jp.}}, Ryoki Fukushima\footnote{\tiny{Institute of Mathematics, University of Tsukuba, 1-1-1 Tennodai, Tsukuba, Ibaraki 305--8571, Japan; ryoki@math.tsukuba.ac.jp.}} and Stefan Junk\footnote{\tiny{Advanced Institute for Materials Research, Tohoku University, 2-1-1 Katahira, Aoba-ku, Sendai, 980-8577 Japan; sjunk@tohoku.ac.jp.}}}
\maketitle

\begin{abstract}
We study the asymptotic behaviour of a version of the one-dimensional Mott random walk in a regime that exhibits severe blocking. We establish that, for any fixed time, the appropriately-rescaled Mott random walk is situated between two environment-measurable barriers, the locations of which are shown to have an extremal scaling limit. Moreover, we give an asymptotic description of the distribution of the Mott random walk between the barriers that contain it.
\end{abstract}

\section{Introduction}\label{sec:intro}

In one dimension, Mott variable-range hopping, which captures the dynamics of an electron in a disordered conduction medium in the Anderson localisation regime, is known to exhibit either diffusive or subdiffusive behaviour, depending on how the parameters of the model are chosen. (Historically, Mott variable-range hopping goes back to the study of electronic processes in non-crystalline materials \cite{Mott1, Mott2}, and the one-dimensional case in particular provides a model of such a process in a wire \cite{Alex, Lee, Serota, Star}.) In \cite{CF}, mathematically rigourous scaling limits for Mott variable-range hopping were obtained in the diffusive regime, as were criteria distinguishing this from the subdiffusive regime. Concerning the latter, anomalous polynomial scaling limits were derived in \cite{CFJ}, wherein the subdiffusivity in question was shown to arise from a `blocking' mechanism, with regions of low conductivity persisting in the limit. The present article extends this study of the subdiffusive regime to a parameter region where the blocking behaviour is even more severe, leading to scaling limits that are described in terms of the so-called extremal processes that arise naturally in the study of sums of independent and identically-distributed random variables whose distributions have slowly-varying tails. (For background on the classical work in the latter area, see \cite{Gnedenko, Lamperti, Res87}.)

We start by introducing the version of the Mott random walk that will be the focus of this article. Let $\omega=(\omega_i)_{i\in\mathbb{Z}}$ be the atoms of a homogeneous Poisson process on $\mathbb{R}$ with unit intensity, conditioned to have an atom at zero (i.e.\ sampled according to the relevant Palm distribution). We assume this is built on a probability space with probability measure $\mathbf{P}$. Moreover, for definiteness concerning the labelling, we will always suppose
\[\dots <\omega_{-2}< \omega_{-1}<\omega_0=0<\omega_1<\omega_2<\cdots.\]
Given a realisation of $\omega$, we define conductances $(c^{\alpha,\lambda}(x,y))_{x,y\in \omega}$ by setting $c^{\alpha,\lambda}(\omega_i,\omega_i)=0$ and
\begin{equation}\label{cdef}
c^{\alpha,\lambda}(\omega_i,\omega_j):=\exp\left(-|\omega_i-\omega_j|^\alpha+\lambda(\omega_i+\omega_j)\right),\qquad \forall i\neq j,
\end{equation}
where $\alpha>1$ and $\lambda\in\mathbb{R}$ are deterministic constants. The associated continuous-time random walk we consider, which will be denoted by $X=(X_t)_{t\geq 0}$, has generator characterised by
\begin{align}\label{generator}
(L^{\alpha,\lambda}f)(\omega_i)&:=\sum_{j\in\Z}\frac{c^{\alpha,\lambda}(\omega_i,\omega_j)}{c^{\alpha,\lambda}(\omega_i)}\left(f(\omega_j)-f(\omega_i)\right),
\end{align}
for bounded $f:\omega\rightarrow\mathbb{R}$, where \[c^{\alpha,\lambda}(\omega_i):=\sum_{j\in\Z}c^{\alpha,\lambda}(\omega_i,\omega_j).\] It is readily checked that, since $\alpha>1$, the random variable $c^{\alpha,\lambda}(\omega_i)$ is finite for every $i\in \mathbb{Z}$, $\mathbf{P}$-a.s. Note that this means the generator at \eqref{generator} is well-defined. Moreover, the Markov chain $X$ has unit mean exponential holding times, which implies it does not explode (i.e.\ remains in $\omega$ for all time), and has invariant measure $\mu$, with
\begin{align}\label{eq:def_invariant}
\mu(\{\omega_i\})=c^{\alpha,\lambda}(\omega_i).
\end{align}
We will typically call $X$ the Mott random walk, and write $P^{\alpha,\lambda}_\omega$ for the law of $X$ started from $0$, conditional on $\omega$; this is the so-called quenched law of $X$.
The corresponding annealed law is obtained by integrating out the randomness of the environment, i.e.\
\[\mathbb{P}^{\alpha,\lambda}:=\int P^{\alpha,\lambda}_\omega\left(\cdot\right)\mathbf{P}(\dd\omega).\]

Our first main result, Theorem~\ref{thm:mainnew1} (see also Remark~\ref{locrem}), establishes that the running supremum and infimum of $X$, i.e.\
\[\overline{X}_t:=\sup_{s\leq t}X_s,\qquad \underline{X}_t:=\inf_{s\leq t}X_s,\]
localize on environment-measurable processes. En route to proving this, we also obtain a similar conclusion for the exceedance times $(\Delta^+_x)_{x\geq0}$ and $(\Delta^-_x)_{x\geq 0}$, defined by setting
\begin{align*}
\Delta^+_{x}&:=\inf\{t\geq 0:\: X_t> x\},\\
\Delta^-_{x}&:=\inf\{t\geq 0:\: X_t<-x\}.
\end{align*}
We note that $(\Delta^+_x)_{x\geq0}$ and $(\Delta^-_{x})_{x\geq 0}$ are \cadlag functions, and represent the right-continuous inverses of $\overline{X}$ and $-\underline{X}$, respectively. As will be explained in more detail after we have given the statements of our results, the long-term behavior of $\overline X $ and $\underline X$ is essentially explained by the nearest-neighbour edge resistances,
\begin{align*}
r^{\alpha,\lambda}(\omega_j,\omega_{j+1}):= c^{\alpha,\lambda}(\omega_j,\omega_{j+1})^{-1}.
\end{align*}
Indeed, as alluded to at the beginning of the section, the key feature of the Mott random walk in the regime that we are studying is blocking, and, intuitively, the family $(r^{\alpha,0}(\omega_{j},\omega_{j+1}))_{j=0,...,k-1}$ describes the height of the barriers that $X$ has to overcome to reach $\omega_k$. (Note that the drift parameter $\lambda$ is not particularly important in determining barrier heights.) The family $(r^{\alpha,0}(\omega_{j},\omega_{j+1}))_{j\in\mathbb{Z}}$ is independent and identically distributed (i.i.d.), with
\begin{align}\label{eq:slow_varying_tail}
\mathbf{P}\left(r^{\alpha,0}(\omega_1,\omega_0)>u\right)=\frac1{L(u)},\qquad \forall u\geq 1,
\end{align}
where
\begin{align}\label{ldef}
L(u)=e^{\log^{1/\alpha}(u)}
\end{align}
is slowly varying as $u\rightarrow\infty$, and has inverse given by
\begin{equation}
\label{eq:inverse}
L^{-1}(u)=e^{\log^{\alpha}(u)}.
\end{equation}
(Note that $L$ and $L^{-1}$ are only defined as above when $u\geq 1$. We extend $L$ to a continuous, strictly increasing function on $[0,\infty)$ by setting $L(u)=u$ for $u\in [0,1)$. The inverse $L^{-1}$ function is then extended in the same way.) Based on the above observation, and in particular using the function $L$ to capture the appropriate scaling, we define environment-measurable processes $(m_{n,+}(x))_{x\geq0}$ and $(m_{n,-}(x))_{x\geq0}$ by setting:
\begin{align}
m_{n,+}(x)&:=\max_{0 \le \omega_j \le xn} \frac{1}{n}L(r^{\alpha,0}(\omega_j,\omega_{j+1})),\label{mnp}\\
m_{n,-}(x)&:=\max_{-xn\leq \omega_j \le0} \frac{1}{n}L(r^{\alpha,0}(\omega_{j-1},\omega_{j})).\label{mnn}
\end{align}
The right-continuous inverses are given by
\begin{eqnarray}
m^{-1}_{n,+}(t)&:=&\inf\{x\geq 0:\:m_{n,+}(x)>t\},\label{mnpinv}\\
m^{-1}_{n,-}(t)&:=&-\inf\{x\geq 0:\:m_{n,-}(x)>t\},\label{mnninv}
\end{eqnarray}
respectively. We are now ready to state our initial conclusion, which starts to provide a picture of the behaviour of the one-dimensional Mott random walk in the extremal, weak drift regime. The subsequent result, Theorem~\ref{thm:mainnew2}, will demonstrate that the environment-measurable processes $(m_{n,+}(x))_{x\geq0}$ and $(m_{n,-}(x))_{x\geq0}$ (and their right-continuous inverses) have explicit, non-trivial distributional limits as $n\rightarrow\infty$. Note that we use $d_U$ to metrise the topology of uniform convergence on compacts, $d_{J_1}$ to metrise the local Skorohod $J_1$ topology and $d_{M_1}$ to metrise the local Skorohod $M_1$ topology; further details are given in Appendix~\ref{metricapp}. Moreover, see Figure~\ref{fig:simu} for output of a simulation that illustrates the result.

\begin{theorem}
\label{thm:mainnew1}
Fix $\alpha>1$ and $\lambda\in\mathbb{R}$.\\
(a) As $n\rightarrow\infty$,
\[d_{U}\left(\left(n^{-1}L\left(n^{-1}\Delta^-_{nx}\right),n^{-1}L\left(n^{-1}\Delta^+_{nx}\right)\right)_{x\geq 0},\left(m_{n,-}(x),m_{n,+}(x)\right)_{x\geq 0}\right)\rightarrow 0\]
in $\mathbb{P}^{\alpha,\lambda/n}$-probability.\\
(b) As $n\rightarrow\infty$,
\[d_{M_1}\left(\left(n^{-1}\underline{X}_{nL^{-1}(n t)},n^{-1}\overline{X}_{nL^{-1}(n t)}\right)_{t\geq 0},\left(m_{n,-}^{-1}(t),m_{n,+}^{-1}(t)\right)_{t\geq 0}\right)\rightarrow 0\]
in $\mathbb{P}^{\alpha,\lambda/n}$-probability.
\end{theorem}

\begin{rem}\label{locrem}
The arguments we give below will further establish that, for each fixed $t>0$, we have localisation of the rescaled versions of $\underline{X}$ and $\overline{X}$ in the sense that
\[\mathbb{P}^{\alpha,\lambda/n}\left(\left(n^{-1}\underline{X}_{nL^{-1}(n t)},n^{-1}\overline{X}_{nL^{-1}(n t)}\right)=\left(m_{n,-}^{-1}(t),m_{n,+}^{-1}(t)\right)\right)\rightarrow 1.\]
See Remark~\ref{locremproof} for a sketch of the proof. This is similar to the one-site localisation in probability for the Bouchaud trap model with slowly-varying traps of \cite[Theorem 1.1]{CM2}. (We recall that, conditional on a trapping environment $(\tau_i)_{i\in\mathbb{Z}}$, which is an i.i.d.\ sequence of $(0,\infty)$-valued random variables, the (symmetric) Bouchaud trap model on $\mathbb{Z}$ is the continuous time nearest-neighbour Markov chain on $\mathbb{Z}$ with jump rates from $i$ to $i\pm1$ given by $\frac{1}{2\tau_i}$. To say that the Bouchaud trap model has slowly-varying traps means that $\mathbf{P}(\tau_i>x)$ is slowly varying as $x\rightarrow\infty$.)
\end{rem}

\begin{figure}
\begin{center}
\begin{tikzpicture}
\node[draw,anchor=south west,inner sep=.5] at (0,0) {\includegraphics[width=.45\textwidth,height=.38\textheight]{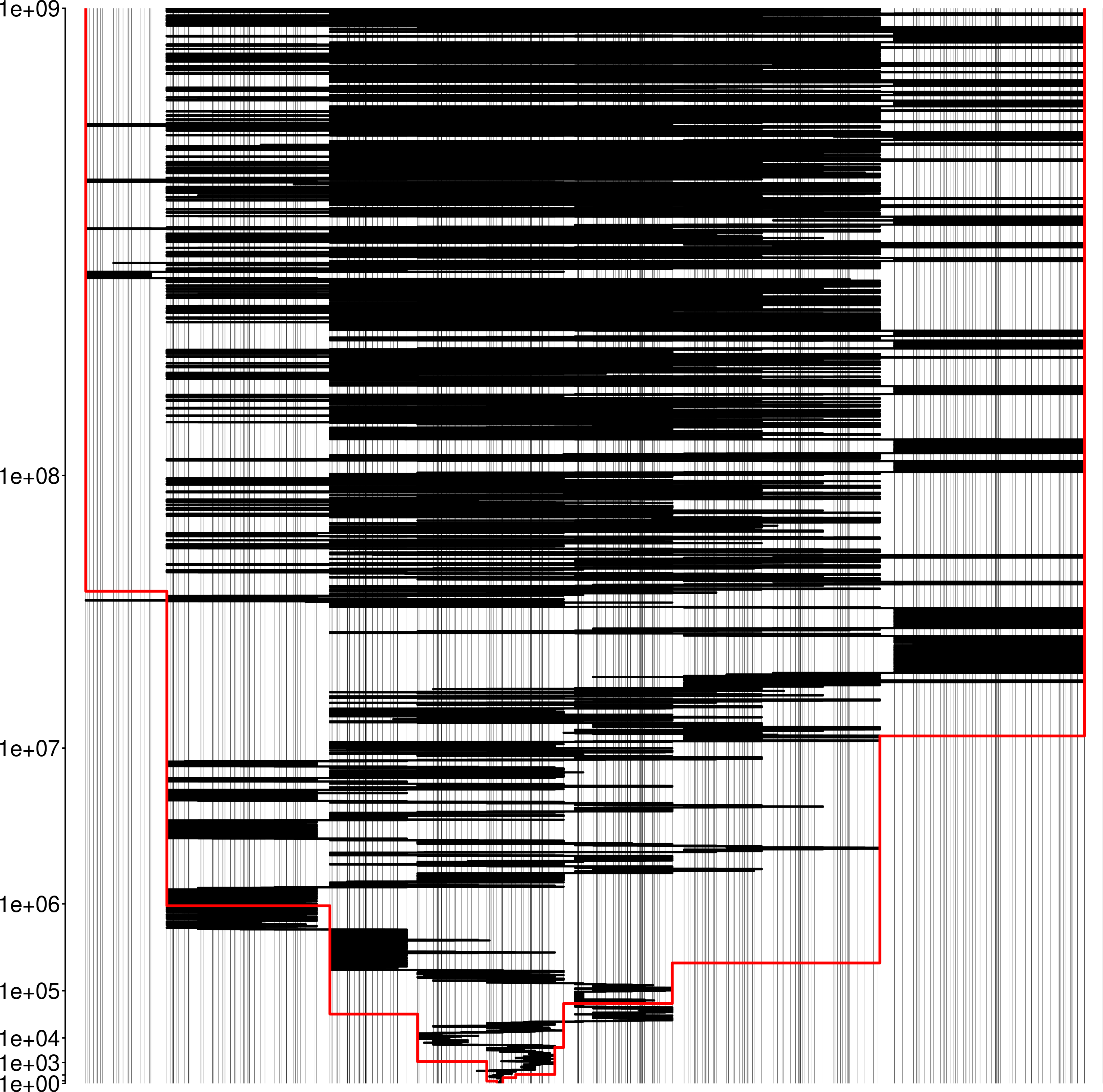}};

\node[draw,anchor=south west,inner sep=.5] at (7.6,0) {\includegraphics[width=.45\textwidth,height=.38\textheight]{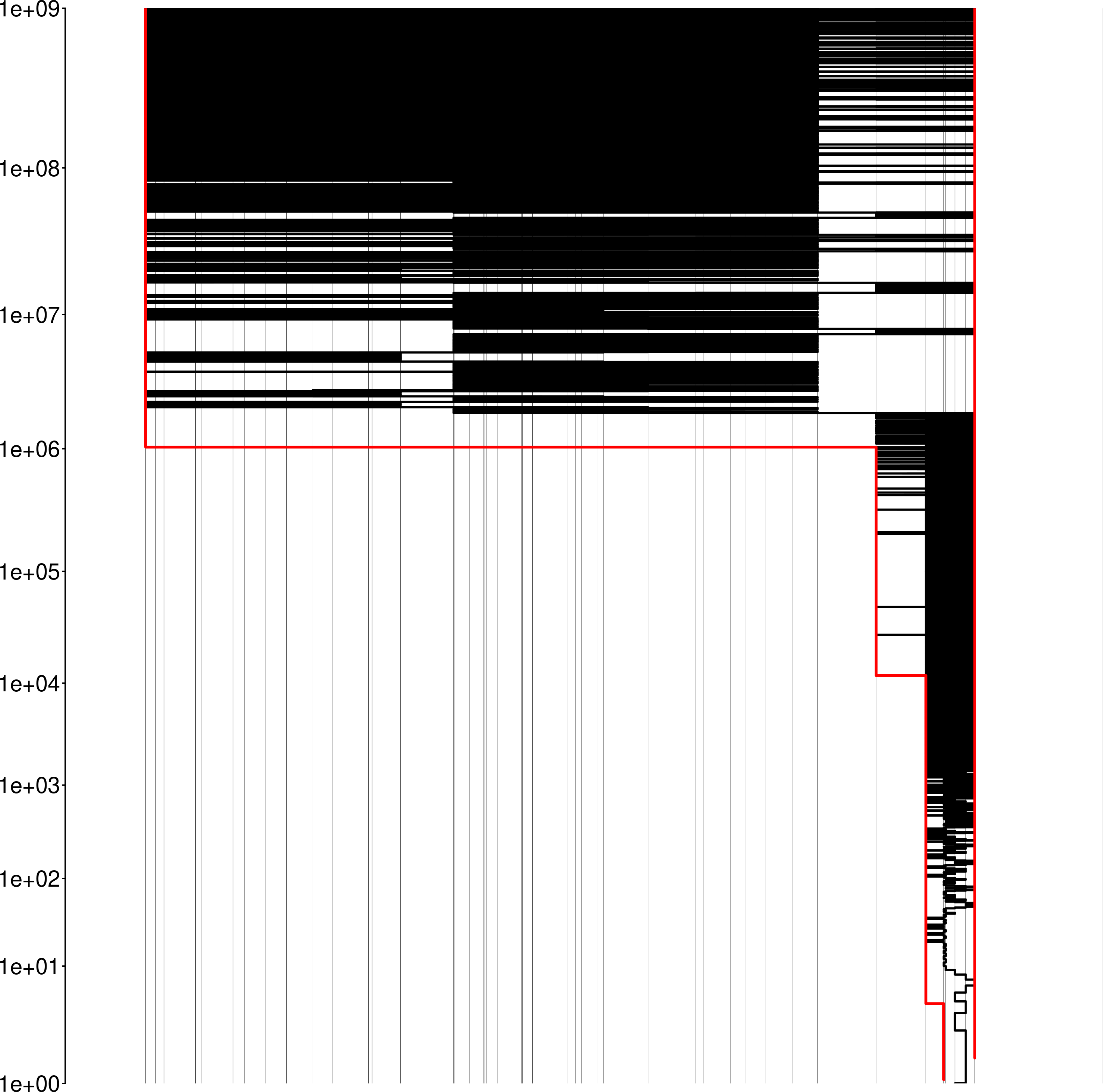}};

\node[draw,anchor=south west,inner sep=.5] at (3.8,-9.3) {\includegraphics[width=.45\textwidth,height=.38\textheight]{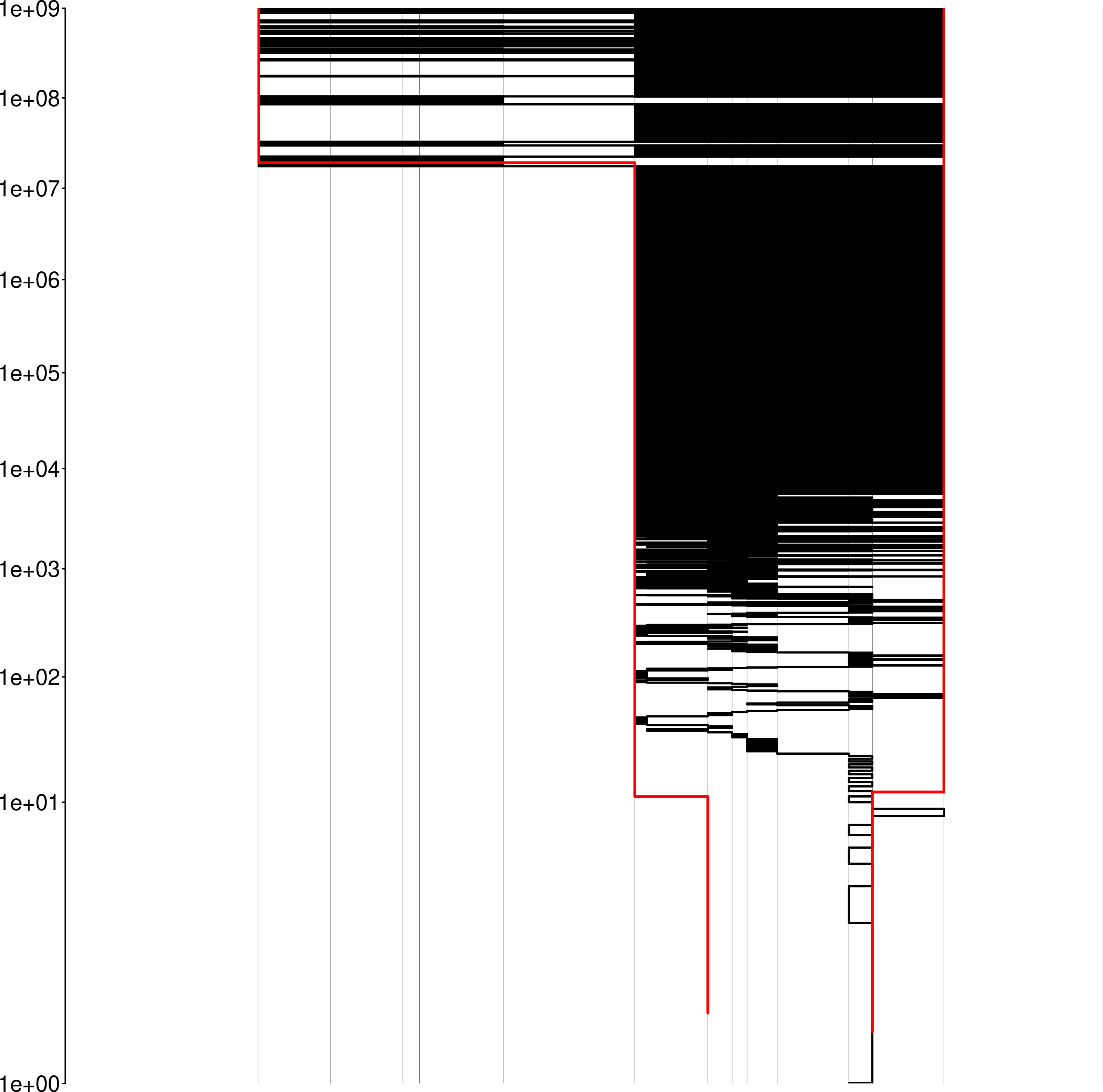}};
\end{tikzpicture}
\end{center}
\caption{Simulation of the Mott random walk $X$ for $10^9$ steps with $\alpha=1.5$ (top left), $\alpha=2.5$ (top right) and $\alpha=3.5$ (bottom), and, in each case, $\lambda=0$. Space runs along the horizontal axis, and time runs upwards, with the vertical axes being rescaled as $t\mapsto L^{-1}(t)$ (note that $L$ depends on $\alpha$). The vertical lines indicate the atoms of $\omega$, and the red curves indicate the $\omega$-measurable processes $m_{n,+}$ and $m_{n,-}$, which describe the space-time region that the Mott random walk is highly likely to both be contained within and fully explore.
\label{fig:simu}}
\end{figure}

To characterise the limits of the processes $(m_{n,+}(x))_{x\geq0}$ and $(m_{n,-}(x))_{x\geq0}$, we introduce a Poisson point process on $\mathbb{R}\times (0,\infty)$ of intensity $v^{-2}\dd x\dd v$, the set of atoms of which will be written $\{(x_i,v_i):i\in\mathbb{Z}\}$. We again suppose this is built on an underlying probability space with probability measure $\mathbf{P}$. We define two associated \cadlag processes by setting, for $x\geq 0$,
\begin{eqnarray}
m_+(x)&:=&\sup\{v_i:\:0\leq x_i\leq x\},\label{mplus}\\
m_-(x)&:=&\sup\{v_i:\:-x\leq x_i\leq 0\}.\label{mminus}
\end{eqnarray}
These are examples of extremal processes, which naturally arise in the study of random variables with slowly-varying tails, see \cite{Kasa} and \cite[Chapter 4]{Res87}, for example. The right-continuous inverses of $m_+$ and $m_-$, which will be denoted $m_+^{-1}$ and $m_-^{-1}$, are defined similarly to \eqref{mnpinv} and \eqref{mnninv}, respectively. Our second main result follows from classical results on extremal processes, and, in conjunction with the previous theorem, give scaling limits for the running supremum and infimum of $X$ (as well as the exceedance times) under the annealed law $\mathbb{P}^{\alpha,\lambda}$.

\begin{theorem}
\label{thm:mainnew2}
Fix $\alpha>1$ and $\lambda\in\mathbb{R}$.\\
(a) As $n\rightarrow\infty$,
\[\mathbf{P}\left(\left(m_{n,-}(x),m_{n,+}(x)\right)_{x\geq 0}\in\cdot\right)\rightarrow\mathbf{P}\left(\left(m_-(x),m_+(x)\right)_{x\geq 0}\in\cdot\right)\]
weakly as probability measures on $D([0,\infty),\mathbb{R}^2)$ with respect to the topology induced by $d_{J_1}$.\\
(b) As $n\rightarrow\infty$,
\[\mathbf{P}\left(\left(m_{n,-}^{-1}(t),m_{n,+}^{-1}(t)\right)_{t\geq 0}\in\cdot\right)\rightarrow\mathbf{P}\left(\left(m_-^{-1}(t),m_+^{-1}(t)\right)_{t\geq 0}\in\cdot\right)\]
weakly as probability measures on $D([0,\infty),\mathbb{R}^2)$ with respect to the topology induced by $d_{M_1}$.
\end{theorem}

In our third main result, we give the asymptotic finite-dimensional distributions of the Mott random walk $X$ under scaling. In particular, Theorem~\ref{thm:mainnew1} and Remark~\ref{locrem} already give that (asymptotically) $n^{-1}{X}_{nL^{-1}(nt)}$ sits between the sites $m_{n,-}^{-1}(t)$ and $m_{n,+}^{-1}(t)$, and the next theorem shows that its density is proportional to $e^{2\lambda x}$ on this interval. Moreover, as the proof strategy will explain, the process mixes suitably quickly that, conditional on the environment, the locations of $n^{-1}{X}_{nL^{-1}(nt)}$ are asymptotically independent for distinct values of $t$. Note that we write $E^{\alpha,\lambda/n}_\omega$ for the expectation under $P^{\alpha,\lambda/n}_\omega$.

\begin{theorem}
\label{thm:mainnew3}
(a) For any collection of times satisfying $0<t_1<t_2<\cdots<t_k$ and continuous bounded functions $f_1,\dots,f_k$, as $n\rightarrow\infty$,
\[\left|E^{\alpha,\lambda/n}_\omega\left(\prod_{i=1}^kf_i\left(n^{-1}{X}_{nL^{-1}(nt_i)}\right)\right)-\prod_{i=1}^k\frac{\int_{m_{n,-}^{-1}(t_i)}^{m_{n,+}^{-1}(t_i)}e^{2\lambda x}f_i(x)dx}{\int_{m_{n,-}^{-1}(t_i)}^{m_{n,+}^{-1}(t_i)}e^{2\lambda x}dx}\right|
\rightarrow0\]
in $\mathbf{P}$-probability.\\
(b) For any collection of times satisfying $0<t_1<t_2<\cdots<t_k$, as $n\rightarrow\infty$,
\[\mathbb{P}^{\alpha,\lambda/n}\left(\left(n^{-1}{X}_{nL^{-1}(nt_i)}\right)_{i=1}^k\in \cdot\right)\rightarrow \mathbf{P}\left(\left(U^\lambda_{t_i}\right)_{i=1}^k\in\cdot\right)\]
where, conditional on $(m_-,m_+)$, $(U^\lambda_{t_i})_{i=1}^k$ is an independent collection of random variables, with $U^\lambda_{t_i}$ having density proportional to $e^{2\lambda x}$ on $(m_-^{-1}(t_i),m_+^{-1}(t_i))$ (and zero elsewhere). In particular, in the case $\lambda=0$, the limiting random variables are simply uniform on the relevant interval.
\end{theorem}

Let us now briefly outline our proof strategy for the above results. As in \cite{CFJ}, the argument used to prove Theorem~\ref{thm:mainnew1} can, at least heuristically, be understood in terms of the scaling of the resistance in the electrical network associated with $(c^{\alpha,\lambda}(x,y))_{x,y\in \omega}$. Towards understanding this, first consider the zero drift ($\lambda=0$) case and recall that the nearest-neighbour edge resistances $r^{\alpha,0}(\omega_i,\omega_{i+1})=c^{\alpha,0}(\omega_i,\omega_{i+1})^{-1}$ have a distributional tail $\mathbf{P}(r^{\alpha,0}(\omega_i,\omega_{i+1})\geq u)=1/L(u)$, where the slowly-varying function $L$ is defined in \eqref{ldef}. Now, if we disregard the non-nearest-neighbour edges, then the effective resistance between $\omega_0$ and $\omega_n$ becomes simply a sum of independent and identically-distributed random variables:
\[R^{\alpha,0}_{\text{\upshape{nn}}}(\omega_0,\omega_n)=\sum_{i=0}^{n-1}r^{\alpha,0}(\omega_{i},\omega_{i+1}).\]
We thus readily deduce from the classical result of \cite[Theorem 2.1]{Kasa} that, as $n\to\infty$,
\begin{equation}\label{rscal}
\left(n^{-1}L\left(R^{\alpha,0}_{\text{\upshape{nn}}}(\omega_0,\omega_{\lfloor nx\rfloor })\right)\right)_{x\geq 0}\rightarrow \left(m_+(x)\right)_{x\geq 0}
\end{equation}
in distribution in the space $D([0,\infty),\mathbb{R})$ with respect to the Skorohod $J_1$ topology, and similarly for resistances along the negative axis. To get from \eqref{rscal} to Theorem~\ref{thm:mainnew1} requires a little more work, as we will next describe.

Firstly, as might be understood from the commute time identity (see, for instance, \cite[Theorem 4.27]{Barlow} for a proof in the context of continuous time random walks), the scaling of the exceedance times $\Delta^\pm$ comes not only from the resistance scaling, but also from the scaling of the invariant measure $\mu$, see \eqref{eq:def_invariant}. Since $\mu(\{\omega_i\})$ has finite expectation, the invariant measure in the present setting scales linearly to Lebesgue measure, which explains the additional factor of $n^{-1}$ in front of $\Delta^\pm$ in Theorem~\ref{thm:mainnew1}(a), as compared to \eqref{rscal}.

\begin{rem}
We note that the contribution from the invariant measure can be neglected if $\alpha>2$. Indeed, \eqref{rscal} and some elementary calculations reveal that
\begin{align*}
\left|\frac 1nL\left(n R_{\text{\upshape{nn}}}^{\alpha,0}(\omega_0,\omega_{\lfloor nx\rfloor })\right)-\frac 1nL\left(R_{\text{\upshape{nn}}}^{\alpha,0}(\omega_0,\omega_{\lfloor nx\rfloor })\right)\right|
\end{align*}
converges to zero in probability if and only if $\alpha>2$. As a consequence, for $\alpha>2$, Theorem~\ref{thm:mainnew1}(b) can be stated more simply as
\[d_{M_1}\left(\left(n^{-1}\underline{X}_{L^{-1}(n t)},n^{-1}\overline{X}_{L^{-1}(n t)}\right)_{t\geq 0},\left(m_{n,-}^{-1}(t),m_{n,+}^{-1}(t)\right)_{t\geq 0}\right)\rightarrow 0\]
in $\P^{\alpha,\lambda/n}$-probability. A similar phenomenon was observed for the Bouchaud trap model with slowly varying traps in \cite[Theorem 1.7]{CM1}, with $\alpha>2$ corresponding to \cite[Assumption 1.4]{CM1}.
\end{rem}

Secondly, in addition to the contribution from the invariant measure, for the statement of the full result of Theorem~\ref{thm:mainnew1} one needs to further consider the effect of non-nearest neighbour edges and the drift term. Whilst both of these add complexity to the argument, the qualitative behaviour of the model remains the same. As for Theorem~\ref{thm:mainnew1}(b), this is obtained from Theorem~\ref{thm:mainnew1}(a) by simply taking appropriate inverses. And, as already noted, Theorem~\ref{thm:mainnew2} is a straightforward application of classical results.

Turning now to Theorem~\ref{thm:mainnew3}, a key point to note in the previous conclusions is that the scaling limits of the exceedance times and the running infimum/supremum are environment measurable, with the Poisson process $\{(x_i,v_i):i\in\mathbb{Z}\}$ capturing the asymptotic behaviour of the locations and sizes of the resistance `barriers' in $\omega$. In particular, the proof of Theorem~\ref{thm:mainnew1} tells us that at time $t$ the random walk $X$ can be found between the first nearest-neighbour edges on the left- and right-hand sides of the origin whose resistances exceed a given threshold. To establish Theorem~\ref{thm:mainnew3} from this, we show that, much more quickly than crossing one of these boundary edges, the random walk $X$ mixes on the part of the space between them. We thus see that the relative position of the random walker between these two edges homogenises to a distribution depending only on the positions of the boundary edges. Moreover, mixing happens sufficiently fast that, on the time scale being considered, the positions of the random walk at distinct time points are asymptotically independent.

In the following remarks, we discuss some natural adaptations/generalisations of our framework.

\begin{rem}
\begin{enumerate}
 \item[(a)] The choice of a unit intensity Poisson process is simply for notational convenience. It would be possible to obtain corresponding results for a general intensity $\rho>0$, with the slowly varying function $L$ being defined by \eqref{eq:slow_varying_tail}.
 \item[(b)] Similar arguments should also apply for more general point processes such that $(\omega_{i+1}-\omega_i)_{i\in\mathbb{Z}}$ are i.i.d.\ and for which the tail of the nearest-neighbour resistance distribution is slowly varying. In this case, one might have to make minor adaptations, replacing the log terms that we use for a generically small deviation, with one suited to the particular slowly-varying function. (Such modifications were managed for the Bouchaud trap model with slowly varying traps in \cite{CM1}.)
\item[(c)] Due to the original physical motivation, it is common in studies of Mott variable-range hopping to include `energy marks', namely a term of the form $-\beta U(E_i,E_j)$ in the exponential defining the conductances at \eqref{cdef}, where $\beta\geq 0$ is the inverse temperature, $U\colon \R\times\R\to[0,1]$ is a symmetric interaction function, and, representing the energy marks, $(E_i)_{i\in\mathbb{Z}}$ is a family of i.i.d.\ random variables on $\mathbb{R}$. The reason for not including such here is again for convenience, and adding them to the model should not affect the arguments in an essential way.
\item[(d)] Whilst we study the constant-speed version of the Mott random walk, with unit mean holding times, another standard choice would be to study the variable-speed version, with generator satisfying
\[(\mathcal{L}^{\alpha,\lambda}f)(\omega_i):=\sum_{j\in\Z}e^{-2\lambda\omega_i}c^{\alpha,\lambda}(\omega_i,\omega_j)\left(f(\omega_j)-f(\omega_i)\right),\]
for bounded $f:\omega\rightarrow\mathbb{R}$. As per \cite[Remarks 1.3 and 3.2]{CFJ}, the same arguments will again apply, but with speed measure being given by $\nu(\{\omega_i\})=e^{2\lambda\omega_i}$, which is in fact easier to study than the invariant measure in the constant-speed case, as defined at \eqref{eq:def_invariant}.
\item[(e)] In the case of a constant drift $\lambda>0$, we expect that the techniques of this article would enable one to show that: as $n\rightarrow\infty$,
\[d_{M_1}\left(\left(n^{-1}{X}_{L^{-1}(n t)}\right)_{t\geq 0},\left(m_{n,+}^{-1}(t)\right)_{t\geq 0}\right)\rightarrow 0\]
in $\mathbb{P}^{\alpha,\lambda}$-probability, and thus
\[\mathbb{P}^{\alpha,\lambda}\left(\left(n^{-1}{X}_{L^{-1}(nt)}\right)_{t\geq 0}\in\cdot\right)\rightarrow\mathbf{P}\left(\left(m_+^{-1}(t)\right)_{t\geq 0}\in\cdot\right).\]
In particular, because in this case the walk moves to the right with little backtracking, the exceedance times $(\Delta^+_x)_{x\geq 0}$ should behave like a sum of i.i.d.\ random variables, each representing the time to cross a resistance barrier (which has order proportional to the size of the barrier). Thus we anticipate the exceedance times to grow asymptotically in the same way as the nearest-neighbour resistance, see \eqref{rscal}, and taking inverses yields the above conclusions. Note that, in contrast to Theorem~\ref{thm:mainnew1}, there is no additional $n$ in the scaling, as the excursions away from the current maximum should not play a significant role in this case.
\item[(f)] It is possible to define the model of this article in higher dimensions. For such, it was shown in \cite{CFP} that the Mott random walk has a non-trivial diffusive scaling limit for any $\alpha>0$.
\end{enumerate}
\end{rem}

We close this discussion by outlining some possible directions for future work. One question that might be asked is about the precise manner in which the rescaled running supremum $\overline X$ moves from a localization site $m_{n,+}^{-1}(t^-)$ to a new maximum $m_{n,+}^{-1}(t)$. Although the Mott random walk has essentially bounded jump length and visits a positive proportion of sites between the two gaps, from Theorems~\ref{thm:mainnew1} and \ref{thm:mainnew2} we know that the running supremum attains these intermediate values for a vanishingly small time. On the other hand, this does not exclude the possibility that there are near-maximal resistance barriers between $m_{n,+}^{-1}(t^-)$ and $m_{n,+}^{-1}(t)$ whose size does not constitute a new record, but which are still large enough to delay the Mott random walk to an extent that is asymptotically visible. More precisely, we formulate the following conjecture about quenched localization for the running supremum.

\begin{conjecture}\label{con:quenched}
There exists an environment-measurable process $(\Gamma_t)_{t\geq 0}$ with $\Gamma_t\subseteq\omega$ and $|\Gamma_t|=2+\lfloor\frac{1}{\alpha-1}\rfloor$ such that
\begin{align}\label{eq:quenched_loc}
\lim_{t\to\infty}P^{\alpha,0}_\omega\left(\overline X_t\in\Gamma_t\right)=1,\qquad \mathbf{P}\mbox{-a.s.}
\end{align}
There is no environment-measurable set satisfying \eqref{eq:quenched_loc} with smaller cardinality than $\Gamma_t$.
\end{conjecture}

\noindent
It is clear that $(\Gamma_t)_{t\geq0}$ must infinitely often contain sets of the form $\{ m_{1,+}^{-1}(t^-), m_{1,+}^{-1}(t)\}$ with $m_{1,+}^{-1}(t^-)\neq m_{1,+}^{-1}(t)$, and thus have size at least two. (Here we refer to $m_{1,+}^{-1}$ rather than $m_{n,+}^{-1}$, as we are now discussing the unrescaled Mott random walk.) The above conjecture says that this estimate on the size of $\Gamma(t)$ is optimal for $\alpha> 2$ and that the transition of the localisation site from $m_{1,+}^{-1}(t^-)$ to $m_{1,+}^{-1}(t)$ happens as a single, large jump in that regime. For $\alpha\in(1,2]$, however, there is, infinitely often, a finite, deterministic number of near-record gaps between $m_{1,+}^{-1}(t^-)$ and $m_{1,+}^{-1}(t)$ where $\overline X$ briefly localizes. The inspiration for Conjecture~\ref{con:quenched} comes from a similar statement that is known to hold for the Bouchaud trap model with slowly varying holding times, see \cite{CM2,M15}. Additionally, we point to \cite[Theorem 1.5]{CM2} for almost-sure bounds on the ratio of a sum of slowly-varying random variables to their maximum, which provides useful insight into the nature of near maxima in sequences of i.i.d.\ random variables with a slowly-varying distributional tail.

In several places in the above discussion, parallels with the Bouchaud trap model have been drawn. This is natural given that the latter model can be seen as something of a dual to the Mott random walk. Indeed, both models admit anomalous scaling limits, with those of the Bouchaud trap model being characterized by trapping in single sites with large holding time, whereas in our setup the random walk itself is never stationary (on an asymptotic scale), but is constrained within regions between large resistance barriers. We refer to \cite[Section 1.5]{CFJ} for a more detailed discussion of this connection. By equipping each atom $(\omega_i)_{i\in\Z}$ independently with a random holding time mean $(\tau_i)_{i\in\Z}$ with suitable tail behavior, we believe that it is possible to construct a random walk that exhibits both the `trapping'  and `blocking' behavior in such a way that an interesting limit process arises, similar to \cite[Theorem~1.8]{CFJ}, which provides such a result in the case of heavy-tailed random variables. Towards proving an extension of this kind, we note that the `blocking' behavior in the Mott random walk is essentially caused by the scaling of the effective resistance, while the `trapping' caused by large holding times emerges from the scaling of the invariant measure. In the slowly-varying regime, both of these objects are supported by extreme values. It is thus reasonable to believe that effective resistance and invariant measure jointly converge to independent scaling limits, which can be used to characterize the potential limit process.

The remainder of the article is organised as follows. In Section~\ref{sec:prelim}, we define features of the environment that occur with high probability and facilitate various random walk estimates being made, and also describe the scaling limit of the invariant measure of the Mott random walk. The key random walk estimates for our argument are then established in Section~\ref{sec:rw}. Finally, in Section~\ref{sec:proof}, we put these pieces together in order to prove Theorems~\ref{thm:mainnew1}, \ref{thm:mainnew2} and \ref{thm:mainnew3}. Concerning notation, constants of form $c,c_i,C,C_i$ may depend on parameters that appear in the argument, such as $\alpha$, $\lambda$, $\delta$, $K$, $\eta$, but will not depend on the level of scaling $n$ or the particular random environment $\omega$. Moreover, we will sometimes use a continuous variable, $x$ say, where a discrete argument is required, with the understanding that it should be treated as $\lfloor x\rfloor$.

\section{Typical structure of the environment}
\label{sec:prelim}

The aim of this section is to describe the typical configuration of the environment in which the Mott random walk evolves. In particular, after introducing some preliminaries concerning the scaling of edges of high resistance in Subsection~\ref{21}, we define an environment with particular features that is seen with high probability in Subsection~\ref{22}. Finally, we derive the asymptotic behaviour of the invariant measure of the Mott random walk in Subsection~\ref{23}.

\subsection{Convergence of the records}\label{21}
In this subsection we focus on the edges of high nearest-neighbour resistance, which will be the critical factor in describing the behaviour of the random walk. We start by defining barrier locations $(a_i^\pm)_{i\in\mathbb{Z}}$ and corresponding barrier sizes $(g_i^\pm)_{i\in\mathbb{Z}}$. We highlight that, although we suppress it from the notation, these are defined for each scale $n$ model separately, and also depend on a deterministic constant $\delta>0$. Specifically, for the barriers on the positive axis, we set $a_0^+:=\mathrm{argmax}_{j\in\{0,1,\dots,\lfloor n\delta\rfloor\}}r^{\alpha,0}(\omega_j,\omega_{j+1})$ (noting that, $\mathbf{P}$-a.s., there are no ties),
\begin{equation}\label{adef}
a_{i+1}^+:=\inf\left\{j\geq a_{i}^+:\:r^{\alpha,0}(\omega_j,\omega_{j+1})>r^{\alpha,0}(\omega_{a_i^+},\omega_{a_{i}^++1})\right\},\qquad \forall i\geq 0,
\end{equation}
and also
\[a_{i-1}^+:=\mathrm{argmax}_{j\in\{0,1,\dots,a_{i}^+-1\}}r^{\alpha,0}(\omega_j,\omega_{j+1}),\qquad \forall i\leq 0,\]
where we define $a_{i-1}^+:=0$ if $a_i^{+}=0$. We moreover write
\begin{equation}\label{gdef}
g_i^+:=r^{\alpha,0}(\omega_{a_i^+},\omega_{a_{i}^++1}),\qquad \forall i\in\mathbb{Z}.
\end{equation}
Proceeding `leftwards' from 0 along the negative axis, we define $(a_i^-)_{i\in\Z}$ and $(g_i^-)_{i\in\Z}$ similarly.

The above barriers are nothing but the record process of the nearest-neighbour edge resistances, with the dependence on $n$ and $\delta$ only determining the parameterisation. In the context of extreme value theory, it is standard to view such a record process as a function of a certain point process. To this end, we introduce a random measure on $\R\times (0,\infty)$ defined by
\begin{align}
  \label{eq:PPPindex}
  \sum_{j\in \Z}\delta_{(j/n, L(r^{\alpha,0}(\omega_j,\omega_{j+1}))/n)}.
\end{align}
We consider $L(r^{\alpha,0}(\omega_j,\omega_{j+1}))$ instead of $r^{\alpha,0}(\omega_j,\omega_{j+1})$ since the distribution
\begin{equation}\label{ltail}
\mathbf{P}\left(L(r^{\alpha,0}(\omega_j,\omega_{j+1}))>u\right)=\frac{1}{u}
\end{equation}
is a well-documented case in the literature of extreme value theory. Indeed, we know from \cite[Corollary 4.19]{Res87} that the measure at \eqref{eq:PPPindex} converges in distribution to the Poisson random measure $\zeta=\sum_{j\in\Z} \delta_{(x_j, v_j)}$ with intensity $v^{-2}\dd x \dd v$. Precisely, we view the measures in question as random measures on $\R \times (0,\infty]$, where $(0,\infty]$ is the compactification of $(0,\infty)$ at $+\infty$, equipped with the vague topology. This compactification makes the function
\begin{align*}
\sum_{j\in\Z} \delta_{(x_j, v_j)}\mapsto (\max\{v_j\colon 0\le x_j \le x\})_{x\geq 0}
\end{align*}
a continuous map into $D([0,\infty),\R)$ when the latter space is equipped with the Skorohod $J_1$ topology and hence, by the continuous mapping theorem,
\[\begin{split}
\left( \frac{1}{n} \max_{0 \le j \le xn} L(r^{\alpha,0}(\omega_j,\omega_{j+1})) \right)_{x\geq 0}\xrightarrow{n\to\infty} (m_+)_{x\geq 0}
\end{split}\]
in distribution, where $m_+$ was defined at \eqref{mplus}, see \cite[Proposition 4.20]{Res87}.

Next, we define $(a_k^{+,\Pois},g_k^{+,\Pois})_{k\in\Z}$ from the above Poisson random measure $\zeta$ as follows:
\[a_0^{+,\Pois}:=\mathrm{argmax}\left\{v_i:\:0\leq x_i\leq \delta\right\},\qquad g_0^{+,\Pois}:=\mathrm{max}\left\{v_i:\:0\leq x_i\leq \delta\right\},\]
and
\[a_{k+1}^{+,\Pois}:=\mathrm{inf}\left\{x_i\geq 0:\:v_i>g_{k}^{+,\Pois}\right\},\qquad \forall k\geq 0,\]
\[a_{k-1}^{+,\Pois}:=\mathrm{argmax}\left\{v_i:\:0\leq x_i<a_{k}^{+,\Pois}\right\},\qquad \forall k\leq 0,\]
where we write $g_k^{+,\Pois}$ for the value of $v_i$ for the unique atom $(x_i,v_i)$ in $\zeta$ with $x_i=a_k^{+,\Pois}$. We define $(a_k^{-,\Pois},g_k^{-,\Pois})_{k\in\Z}$ on the negative side similarly. From basic properties of the Poisson random measure, one can check that, almost-surely, these quantities are well-defined and satisfy
\begin{equation}\label{aok}
\dots<a_1^{-,\Pois}<a_0^{-,\Pois}<a_{-1}^{-,\Pois}<\dots<0<\dots<a_{-1}^+<a_0^{+,\Pois}<a_1^{+,\Pois}<\dots,
\end{equation}
\begin{equation}\label{gok}
0<\dots<g_{-1}^{*,\Pois}<g_0^{*,\Pois}<g_1^{*,\Pois}<\dots, \text{ for }*\in\{+,-\},
\end{equation}
\begin{equation}\label{gsep}
  \min\left\{|g_k^{+,\Pois}-g_l^{-,\Pois}|\colon -M\le k, l \le M\right\}>0, \text{ for any }M\in\N,
\end{equation}
and also
\begin{equation}\label{adiv}
a_k^{\pm,\Pois}\xrightarrow{k\rightarrow\infty} \pm \infty, \quad g_k^{\pm,\Pois}\xrightarrow{k\rightarrow\infty} \infty,\quad g_k^{\pm,\Pois}\xrightarrow{k\rightarrow-\infty}0,
\end{equation}
\begin{equation}\label{gto0}
g_0^{\pm,\Pois}\xrightarrow{\delta\downarrow 0} 0.
\end{equation}

Finally, by applying the Skorohod representation theorem, we can construct a coupling so that \eqref{eq:PPPindex} converges to $\zeta$ almost surely. Under this coupling, the atoms of \eqref{eq:PPPindex} away from $v=0$ converge to those of $\zeta$ and, as a consequence, the jumps of the extremal process $(m_{n,+}(x))_{x>0}$ away from $x=0$ converge to those of $m_+$. (Recall the definition of $m_{n,+}$ from \eqref{mnp}.) Since the same argument applies to $m_{n,-}$ and $m_-$ (as defined at \eqref{mnn} and \eqref{mminus}, respectively), we have
\begin{equation}\label{agconvindex}
\left(n^{-1}{a_k^\pm},n^{-1}L\left(g_k^\pm\right)\right)_{k\in\Z}\rightarrow \left(a_k^{\pm,\Pois},g_k^{\pm,\Pois}\right)_{k\in\Z}
\end{equation}
as $n\to\infty$. We remark that it might have been more natural to write $g_k^{+,\Pois}$ for the value of $L^{-1}(v_i)$ for the unique atom $(x_i,v_i)$ in $\zeta$ with $x_i=a_k^{+,\Pois}$, so that the above limit involves $L(g_k^{+,\Pois})$; we adopt the above convention for brevity of notation later in the paper. Furthermore, it will be useful later to write the above convergence in terms of the spatial position of discrete records, rather than their index. In particular, using the
functional law of large numbers, i.e.\
\begin{align*}
\sup_{t\in [0,T]}|\tfrac{1}{n}\omega_{\lfloor tn\rfloor}-t|\xrightarrow{n\to\infty} 0,\quad \mathbf P\text{-a.s.},
\end{align*}
for any $T>0$, the convergence at \eqref{agconvindex} yields that, almost surely as $n\to\infty$,
\begin{equation}\label{agconv}
\left(n^{-1}\omega_{a_k^\pm},n^{-1}L\left(g_k^\pm\right)\right)_{k\geq 1}\rightarrow \left(a_k^{\pm,\Pois},g_k^{\pm,\Pois}\right)_{k\geq 1}.
\end{equation}

\subsection{Features of a typical environment}\label{22}

In this subsection, we collect various `nice' features of a typical environment that we will use in the proof of our main results, and prove that they indeed hold with high probability. To this end, we define, for $k\in\Z$, the first exceedance of $g_k^+$ by a record on the left side of the origin,
\begin{align}\label{eq:def_ex}
\ex(k):=\inf\{l\in\Z:g^-_l\geq g^+_k\}
\end{align}
Note that by construction
\[\gtt{k}=\max_{j\in\{\at{k},\dots,-1\}}r^{\alpha,0}(\omega_j,\omega_{j+1})\]
represents the size of the largest barrier between $\omega_{\at{k}}$ and the origin. Let $\ell_i$ be the $i$-th iterate of the logarithm, for example, $\ell_2(n)=\log\log n$, and let
\[N:=2+\left\lfloor\frac{1}{\alpha-1}\right\rfloor\]
be the constant appearing in Conjecture~\ref{con:quenched}.

Applying the above notation, we now introduce an event that we will use in Section~\ref{sec:rw}, which ensures that the environment behaves typically.

\begin{definition}\label{def:A}
For $\delta>0$ and $K,n\in\N$, let $A^{\delta,K}_n$ be the set of those $\omega$ that satisfy the following properties.
\begin{itemize}
\item The record locations are at (approximately) linear distance from the origin:
\begin{align}
\omega_{a_k^+},-\omega_{\at{k}}\leq n\ell_3(n),\qquad\text{ for any }k=0,...,K.\label{eq:space_scaling}
\end{align}
\item New record values are much larger than any previous record:
\begin{align}
\frac{g_{k}^+}{g_{k-1}^+ + \gtt{k}}> e^{\log^{\alpha-1} n/\ell_2(n)},\qquad\text{ for any }k=1,...,K.\label{eq:separation}
\end{align}
\item The contribution from non-record values is bounded in terms of the last record value:
\begin{align}
&{\textstyle \sum_{l=1}^{a_k^+} }r^{\alpha,0}(\omega_{l-1},\omega_{l})\le Ng_{k-1}^+,&\text{for any }k=1,...,K,\label{eq:NN+}\\
&{\textstyle \sum_{l=\at{k}+1}^0} r^{\alpha,0}(\omega_{l-1},\omega_l)\le N \gtt{k},&\text{for any }k=1,...,K.\label{eq:NN-}
\end{align}
\item The atoms of the invariant measure (recall \eqref{eq:def_invariant}) are (approximately) bounded:
\begin{align}
e^{-\log^{\alpha+1}n}\leq c^{\alpha,\lambda/n}(\omega_j)\leq \ell_1(n)^2,\qquad\text{ for all }j=-n\ell_3(n),...,n\ell_3(n)\label{eq:pointwise}.
\end{align}
\item The total mass of the invariant measure grows (approximately) linearly:
\begin{align}
c_1n \le {\textstyle \sum_{l=\at{k}}^{a_k^+}} c^{\alpha,\lambda/n}(\omega_l) \le c_2n \ell_1(n)^2.\label{eq:mass_control}
\end{align}
\end{itemize}
The event $\tilde A_n^{\delta,K}$ is defined in the same way, with the role of positive and negative axis reversed.
\end{definition}

The next result shows that these estimates hold with high probability.

\begin{proposition}
  \label{prop:Awhp}
  For any $\delta>0$ and $K\in\N$, $\lim_{n\to\infty}\mathbf{P}(A^{\delta,K}_n)=1$.
\end{proposition}
\begin{proof}
It follows from~\eqref{aok}, \eqref{adiv} and \eqref{agconv} that
\begin{align*}
    \lim_{M\to\infty}\liminf_{n\to\infty}\mathbf{P}\left(\{\omega_{a_k^+}/n\}_{0\le k \le K} \subset [M^{-1}, M] \text{ and } \{\omega_{\at{k}}/n\}_{0\le k \le K}\subset [-M,-M^{-1}]\right)=1.
\end{align*}
This implies that~\eqref{eq:space_scaling} holds with high probability.

Next, note that by \eqref{gok}, \eqref{adiv} and \eqref{agconv}, the probability of $g_{-M}^-\leq \gtt{0} \leq \gtt{K}\leq g_M^-$ can be made arbitrarily large by making $M$ large. Applying this estimate in conjunction with \eqref{gok}, \eqref{gsep} and \eqref{agconv} again, we obtain
\begin{align}
  \label{eq:fin-vol2}
  \lim_{\epsilon\to 0}\liminf_{n\to\infty} \mathbf{P}
    \begin{pmatrix}
      \min\{|L(g_k^+)/n-L(\gtt{l})/n|\colon 0\le k,l \le K\}\ge \epsilon,\\
      \min\{|L(g_k^+)/n-L(g^+_{k-1})/n|\colon 1\le k \le K\}\ge \epsilon,\\
      \{L(g_k^+)/n, L(\gtt{k})/n\}_{k=0}^K \subset [\epsilon,\epsilon^{-1}]
    \end{pmatrix}
    =1.
\end{align}
We shall assume that the event in \eqref{eq:fin-vol2} holds for some $\epsilon <1$. Then, for $1\leq k\leq K$, we have
\begin{align*}
  L(g_k^+)/n \ge \max\{L(g_{k-1}^+)/n, L(\gtt{k})/n\} +\epsilon,
\end{align*}
and thus, recalling \eqref{eq:inverse} and using the convexity of $x\mapsto x^\alpha$, it can be checked that
\[  g_k^+ \ge \exp\left\{\log^\alpha (L(g_{k-1}^+)+\epsilon n)\right\}\ge g_{k-1}^+ \exp\left\{\alpha \log^{\alpha-1} (L(g_{k-1}^+)) \log\left(1+\epsilon n/L(g_{k-1}^+)\right)\right\}.\]
Using the inequality $\log (1+x) \ge x/2$ for $x\in[0,1]$ in the last logarithmic factor and noting that we have $\epsilon n \le  L(g_{k-1}^+)\le \epsilon^{-1}n$ on the event in \eqref{eq:fin-vol2}, we find
\begin{align*}
  g_k^+\ge g_{k-1}^+\exp\left(\tfrac{\alpha \epsilon^2}{2} \log^{\alpha-1} (\varepsilon n)\right),
\end{align*}
and similarly $g_k^+\ge \gtt{k}\exp(\tfrac{\alpha \epsilon^2}{2} \log^{\alpha-1}(\varepsilon n))$. The second assertion \eqref{eq:separation} follows from these bounds together with \eqref{eq:fin-vol2}.

To prove \eqref{eq:NN+} and \eqref{eq:NN-}, let $\nu_k\in \N$ be such that $a^+_k$ is the index of the $\nu_k$-th record in the sequence $(r^{\alpha,0}(\omega_l,\omega_{l+1}))_{l\ge 0}$. Then we have
\begin{equation}
  \label{eq:CMused}
\mathbf{P}\left({\textstyle \sum_{l=1}^{a_k^+}}r^{\alpha,0}(\omega_{l-1},\omega_{l}) \geq N g_{k-1}^+\right) \le \sum_{j\ge J}\mathbf{P}\left({\textstyle \sum_{l=1}^{a_k^+}}r^{\alpha,0}(\omega_{l-1},\omega_{l}) \geq N g_{k-1}^+, \:\nu_k=j\right)
  + \mathbf{P}(\nu_k< J).
\end{equation}
By \cite[Proposition 4.9]{Res87}, the process $m^+$ has infinitely many jumps near the origin. This and \eqref{agconv} show that $\nu_k \to \infty$ as $n\to\infty$ with high probability. Thus the second term on the right-hand side of~\eqref{eq:CMused} tends to zero as $n\to\infty$. On the other hand, \cite[Lemma 3.8]{CM2} gives a bound for the first term on the right-hand side of~\eqref{eq:CMused} that is uniform in $n$ and tends to zero as $J\to\infty$. From these facts, \eqref{eq:NN+} follows. The bound at \eqref{eq:NN-} can be proved in the same way.

We turn to proving the upper bound in \eqref{eq:pointwise}. We begin by rewriting
  \begin{align}\label{eq:c_rep}
    c^{\alpha,\lambda/n}(\omega_j)&=\sum_{l\in\Z\setminus\{j\}}e^{\lambda(\omega_j+\omega_l)/n-|\omega_l-\omega_j|^\alpha}=e^{2\lambda\omega_j/n}\sum_{l\in\Z\setminus\{j\}}e^{\lambda(\omega_l-\omega_j)/n-|\omega_l-\omega_j|^\alpha}.
  \end{align}
The point process $(\omega_l-\omega_j)_{l\in\Z\setminus\{j\}}$ is a Poisson point process with unit intensity, independent of $\omega_j$. Therefore, by using Campbell's theorem for Poisson processes, we get
\begin{align*}
  \mathbf{E}\left[\exp\left(2\sum_{l\in\Z\setminus\{j\}}e^{\lambda(\omega_l-\omega_j)/n-|\omega_l-\omega_j|^\alpha}\right)\right]
  =\exp\left(\int \left(e^{2e^{\lambda x/n -|x|^\alpha}}-1\right) \dd x\right).
\end{align*}
One can readily verify that the right-hand side is bounded by a constant  $c_\lambda>0$ that depends only on $\lambda$. By using Chebyshev's inequality, we can thus deduce
\begin{align*}
  \mathbf{P}\left(\sum_{l\in\Z\setminus\{j\}}e^{\lambda(\omega_l-\omega_j)/n-|\omega_l-\omega_j|^\alpha} \le \log n\right)\ge 1-c_\lambda n^{-2},
\end{align*}
and, by applying the union bound, one further obtains that the above events hold for all $j\in[-n\ell_3(n),n\ell_3(n)]$ with high probability. In addition, the law of large numbers implies that $\max_{j=-n\ell_3(n),...,n\ell_3(n)}|\omega_j| = \max\{-\omega_{-n\ell_3(n)},\omega_{n\ell_3(n)}\}\leq 2n\ell_3(n)$ with high probability. Combining these bounds with \eqref{eq:c_rep}, we find that
\begin{align}
  \label{eq:mass3/2}
  c^{\alpha,\lambda/n}(\omega_j)&
  \le e^{4|\lambda| \ell_3(n)}\log n \le \ell_1(n)^{3/2}
\end{align}
holds for all $j\in[-n\ell_3(n),n\ell_3(n)]$ with high probability as $n\to\infty$. For the lower bound in \eqref{eq:pointwise} we observe, for all $i=-n\ell_3(n),...,n\ell_3(n)$,
\begin{align*}
c^{\alpha,\lambda/n}(\omega_i)=\sum_{j\in\Z\setminus\{i\}} e^{-|\omega_i-\omega_j|^\alpha+\lambda(\omega_i+\omega_j)/n}\geq  e^{-|\omega_i-\omega_{i+1}|^\alpha-2|\lambda|\max\{-\omega_{-n\ell_3(n)},\omega_{n\ell_3(n)}\}/n}.
\end{align*}
Hence, the desired lower bound holds, for $n$ large enough, on
\begin{align*}
\left\{\max\{-\omega_{-n\ell_3(n)},\omega_{n\ell_3(n)}\}\leq 2n\ell_3(n)\right\}\cap\left\{\max_{i=-n\ell_3(n),...,n\ell_3(n)}|\omega_{i+1}-\omega_i|\leq 2\log n\right\}.
\end{align*}
The probability of the first event converges to one by the law of large numbers. To see that the second event also has large probability, we apply the union bound together with
\begin{align*}
\mathbf{P}(\omega_{i+1}-\omega_i\geq 2\log n)=n^{-2}.
\end{align*}

Finally, we prove \eqref{eq:mass_control}. By the law of large numbers and \eqref{eq:space_scaling}, we know that $-2n\ell_3(n)\le \at{k}< a_k^+\le 2n\ell_3(n)$ holds with high probability. The upper bound follows from this and \eqref{eq:mass3/2}. For the lower bound, note first that
\begin{align*}
  \sum_{l=\at{k}}^{a_k^+} c^{\alpha,\lambda/n}(\omega_l)
  \ge \sum_{l=0}^{\lfloor \delta n \rfloor} c^{\alpha,\lambda/n}(\omega_l)
\end{align*}
since $a_k^+ \ge \lfloor \delta n \rfloor$. Using the law of large numbers again, we may assume $\omega_{\lfloor \delta n \rfloor} \le 2\delta n$. Then the above sum is bounded from below by a constant multiple of $\sum_{l=0}^{\lfloor \delta n \rfloor} c^{\alpha,0}(\omega_l)$. The proof of Lemma~\ref{lem:quantitativeLLN} in the next subsection shows that the last sum is bounded from below by $cn$.
\end{proof}

Next, let us introduce an event that we will use in Section~\ref{sec:proof} to ensure that the record values, which correspond to the times where the random walk overcomes the corresponding gaps, are bounded away from any fixed time.

\begin{definition}
  \label{def:E}
  For $\delta, \eta>0$, $K\in\N$ and $t>0$, let $E_n^{\delta,\eta,K}$ be the set of all $\omega\in A_n^{\delta,K}$ that satisfy
\[n^{-1}L(\gtt{k}),n^{-1}L(g_{k-1}^+)\leq t-\eta<t+\eta\leq n^{-1}L(g_{k}^+)\]
  for some $k\in\{1,\dots,K\}$. Analogously, the event $\tilde{E}_n^{\delta,\eta,K}$ is defined by reversing the roles of the positive and negative axes.
\end{definition}
Note that the dependence on $t$ is suppressed from $E_n^{\delta,\eta,K}$ to keep the notation compact.
\begin{proposition}
  \label{prop:Ewhp}
For fixed $\varepsilon,t>0$, it is possible to choose $\delta, \eta$ small enough and $K$ large enough so that
\begin{align}
  \label{eq:Ewhp}
    \liminf_{n\to\infty}\mathbf{P}\left({E}_n^{\delta,\eta,K}\cup \tilde{E}_n^{\delta,\eta,K}\right)\geq 1-\varepsilon.
\end{align}
\end{proposition}
\begin{proof}
Observe that we can write $E_n^{\delta,\eta,K}$ as $E_n^{\delta,\eta,K}(1)\cap E_n^{\delta,\eta,K}(2)\cap E_n^{\delta,\eta,K}(3)\cap A_n^{\delta,K}$, where
\begin{align*}
E_n^{\delta,\eta,K}(1)&:=\left\{n^{-1}L(g_K^+)\geq t+\eta,\:n^{-1}L(g_0^+)\leq t-\eta\right\},\\
E_n^{\delta,\eta,K}(2)&:=\left\{n^{-1}L(g_k^+)\not\in(t-\eta,t+\eta)\mbox{ for any }k\in \{1,\dots,K-1\}\right\},\\
E_n^{\delta,\eta,K}(3)&:=\left\{m_{n,-}(m_{n,-}^{-1}(t-\eta))> m_{n,+}(m_{n,+}^{-1}(t-\eta))\right\}.
\end{align*}
(Note in particular that, in combination with the other events, the event $E_n^{\delta,\eta,K}(3)$ is equivalent to $n^{-1}L(\gtt{k})\leq t-\eta$, where $k$ is the unique index in $\{1,\dots,K\}$ such that $n^{-1}L(g_{k-1}^+)\leq t-\eta\leq t+\eta\leq n^{-1}L(g_k^+)$.) Analogously, reversing the roles of the positive and negative axes, we can write $\tilde{E}_n^{\delta,\eta,K}= \tilde{E}_n^{\delta,\eta,K}(1)\cap \tilde{E}_n^{\delta,\eta,K}(2)\cap\tilde{E}_n^{\delta,\eta,K}(3) \cap \tilde{A}_n^{\delta,K}$.

To prove \eqref{eq:Ewhp}, we consider the bound:
\begin{align*}
\mathbf{P}\left(\left({E}_n^{\delta,\eta,K}\cup \tilde{E}_n^{\delta,\eta,K}\right)^c\right)&\leq \mathbf{P}\left(({A}_n^{\delta,K})^c\right)
+\mathbf{P}\left((\tilde{A}_n^{\delta,K})^c\right)
+\mathbf{P}\left({E}_n^{\delta,\eta,K}(1)^c\right)
+\mathbf{P}\left(\tilde{E}_n^{\delta,\eta,K}(1)^c\right)\\
&\quad+\mathbf{P}\left({E}_n^{\delta,\eta,K}(2)^c\right)
+\mathbf{P}\left(\tilde{E}_n^{\delta,\eta,K}(2)^c\right)
+\mathbf{P}\left({E}_n^{\delta,\eta,K}(3)^c\cap\tilde{E}_n^{\delta,\eta,K}(3)^c\right).
\end{align*}
Now, the final term here is clearly zero since ${E}_n^{\delta,\eta,K}(3)^c\cap\tilde{E}_n^{\delta,\eta,K}(3)^c$ is the empty set. Next, we have from \eqref{agconv} that
\[\limsup_{n\rightarrow\infty} \mathbf{P}\left({E}_n^{\delta,\eta,K}(1)^c\right)\leq\mathbf{P}\left(g_0^{+,\Pois}> t-\eta\right)+\mathbf{P}\left(g_K^{+,\Pois}< t+\eta\right),\]
which, by \eqref{adiv} and \eqref{gto0}, can be made arbitrarily small by taking $\delta$ small and $K$ large. We can deal with $\mathbf{P}(\tilde{E}_n^{\delta,\eta,K}(1)^c)$ similarly. For fixed $\delta$ and $K$, the terms $\mathbf{P}(({A}_n^{\delta,K})^c)$ and $\mathbf{P}((\tilde{A}_n^{\delta,K})^c)$ converge to zero as $n\rightarrow\infty$ by Proposition~\ref{prop:Awhp}. Finally, again for fixed $\delta$ and $K$, we have from \eqref{agconv} that
\[\limsup_{n\rightarrow\infty} \mathbf{P}\left({E}_n^{\delta,\eta,K}(2)^c\right)=
\mathbf{P}\left(g_k^{+,\Pois}\in[t-\eta,t+\eta]\:\mbox{ for some }k\in\{1,\dots,K-1\}\right).\]
That the right-hand side here converges to zero as $\eta\rightarrow 0$ is a straightforward consequence of the fact that, with probability one, the Poisson random measure $\zeta$ defined in Subsection~\ref{21} has no atoms at a fixed level $t$. Since the term $\mathbf{P}(\tilde{E}_n^{\delta,\eta,K}(2)^c)$ can be dealt with in the same way, this completes the proof.
\end{proof}

\subsection{Convergence of the invariant measure}\label{23}

In this subsection, we show that the measure defined by
\begin{align}\label{mundef}
  \mu_n^{\alpha,\lambda,\omega}:=\sum_{k\in\Z}\frac{1}{n}c^{\alpha,\lambda/n}(\omega_k)\delta_{\omega_k/n}
\end{align}
converges to $e^{2\lambda x}\mathbf{E}[c^{\alpha,0}(\omega_0)]\dd x$. The specific statement that we will need for later is as follows.

\begin{proposition}
  \label{prop:inv_meas}
  For any $t>0$ and $f\in C_b(\R)$,
  \begin{align*}
    \int_{m_{n,-}^{-1}(t)}^{m_{n,+}^{-1}(t)}f(x)\mu_n^{\alpha,\lambda,\omega}(\dd x)-
    \int_{m_{n,-}^{-1}(t)}^{m_{n,+}^{-1}(t)}e^{2\lambda x}f(x)\mathbf{E}[c^{\alpha,0}(\omega_0)]\dd x\xrightarrow{n\to\infty} 0
  \end{align*}
in $\mathbf{P}$-probability.
\end{proposition}

Towards checking this, we start with the following simple result in the case $\lambda=0$.

\begin{lemma}
  \label{lem:quantitativeLLN}
  For any real numbers $a<b$,
  \begin{align*}
    \mu^{\alpha,0,\omega}_n([a,b])-\mathbf{E}[c^{\alpha,0}(\omega_0)](b-a)\xrightarrow{n\to\infty} 0
  \end{align*}
in $\mathbf{P}$-probability.
\end{lemma}

\begin{proof}
Note first that
\begin{align*}
  \mu^{\alpha,0,\omega}_n([\omega_{an}/n,\omega_{bn}/n])=\sum_{k\in [an,bn]}\frac{1}{n}c^{\alpha,0}(\omega_k)
\end{align*}
is an empirical mean of the stationary and ergodic sequence $(c^{\alpha,0}(\omega_k))_{k\in\mathbb{Z}}$, and hence it converges to the constant $(b-a)\mathbf{E}[c^{\alpha,0}(\omega_0)]$ as $n\to\infty$, $\mathbf{P}$-almost surely. Next, since $\lim_{n\to\infty}\omega_{an}/n= a$ and $\lim_{n\to\infty}\omega_{bn}/n=b$, $\mathbf{P}$-almost surely, it readily follows that
\begin{align*}
\mu^{\alpha,0,\omega}_n([a,b])-\mu^{\alpha,0,\omega}_n([\omega_{an}/n,\omega_{bn}/n])\xrightarrow{n\to\infty} 0
\end{align*}
in $\mathbf{P}$-probability.
\end{proof}

\begin{proof}[Proof of Proposition~\ref{prop:inv_meas}]
For convenience, in this proof, we will assume that $\lambda>0$. The proof for $\lambda\leq 0$ follows the same argument. Observing that
\begin{align*}
c^{\alpha,\lambda/n}(\omega_k)=e^{2\lambda\omega_k/n}\sum_{l\in\Z\setminus\{k\}}e^{\lambda(\omega_l-\omega_k)/n-|\omega_l-\omega_k|^\alpha},
\end{align*}
we define the `stationary part' of the invariant measure to be
\begin{align*}
\tilde{\mu}^{\alpha,\lambda,\omega}_n:=\sum_{k\in\Z}\frac{1}{n}\delta_{\omega_k/n}\sum_{l\in\Z\setminus\{k\}}e^{\lambda(\omega_l-\omega_k)/n-|\omega_l-\omega_k|^\alpha}.
\end{align*}
With this notation, we have $\int f(x)  \mu^{\alpha,\lambda,\omega}_n(\dd x)=\int e^{2\lambda x}f(x)  \tilde{\mu}^{\alpha,\lambda,\omega}_n(\dd x)$.

Since $m_{n,\pm}$ converge to the processes $m_\pm$ whose records satisfy \eqref{aok}--\eqref{gto0}, the event that
\begin{align*}
m_{n,+}^{-1}(t)&\in \{a_i^+/n\colon 0\le i \le K\} \subseteq [M^{-1}, M],\\
m_{n,-}^{-1}(t)&\in \{a_j^-/n\colon 0\le j \le K\} \subseteq [-M,-M^{-1}]
\end{align*}
is readily checked to hold with arbitrarily high probability, uniformly in $n$, when $\delta$ is chosen suitably small and $K$ and $M$ suitably large. Thus it suffices to show that, for any fixed $0\le i,j \le K$ and $f\in C_b([-M,M])$,
\begin{align}
  \label{eq:meas_goal}
  \int_{a_j^-/n}^{a_i^+/n}e^{2\lambda x}f(x)\tilde{\mu}_n^{\alpha,\lambda,\omega}(\dd x)-
  \int_{a_j^-/n}^{a_i^+/n}e^{2\lambda x}f(x)\mathbf{E}[c^{\alpha,0}(\omega_0)]\dd x \xrightarrow{n\to\infty} 0
\end{align}
in $\mathbf{P}$-probability. We will first show that $\tilde\mu_n^{\alpha,\lambda,\omega}$ can be replaced by $\mu_n^{\alpha,0,\omega}$, see \eqref{eq:inv_meas1st} below. Once that is achieved, the result follows by dividing $[a_j^-/n,a_i^+/n]$ into small intervals and applying Lemma~\ref{lem:quantitativeLLN} to each interval separately. To ease notation, we write $f_\lambda(x)=e^{2\lambda x}f(x)$, which is continuous and bounded on $[-M,M]$.

Towards the first aim, we take $\theta\in (0,1)$ and use the bound: for $j\in\mathbb{N}$ and $J>0$,
\begin{align*}
  \mathbf{P}(|\omega_j-j|>J)
  &= \mathbf{P}(\omega_j-j>J)+\mathbf{P}(\omega_j-j<-J)\\
  &\le e^{-\theta(j+J)}\mathbf{E}[e^{\theta\omega_1}]^j+e^{\theta(j-J)}\mathbf{E}[e^{-\theta\omega_1}]^j\\
  &\le e^{-\theta(j+J)}\left(\frac{1}{1-\theta}\right)^j+e^{\theta(j-J)}\left(\frac{1}{1+\theta}\right)^j.
\end{align*}
It follows that
\[\mathbf{P}(|\omega_j-j|>J) \le
  \begin{cases}
    \exp\{-cn^{1/3}\}, & \text{when $0\leq j \le 2Mn$ and $J=n^{2/3}$,} \\
    \exp\{-cj\}, & \text{when $j > 2Mn$ and $J=\frac{j}{4}$,}
  \end{cases}\]
where we have taken $\theta=n^{-1/3}/M^2$ in the first case and $\theta=\frac{1}{10}$ in the second. Similar bounds hold when $j\leq 0$. Combining these with a union bound, we can show that the event
\begin{align}
\label{eq:MDP}
\bigcap_{j\in[-2Mn, 2Mn]}\{|\omega_j-j| \le n^{2/3}\}\cap \bigcap_{j\not\in[-2Mn, 2Mn]}\left\{|\omega_j-j| \le \frac{j}{4}\right\}
\end{align}
has probability larger than $1-\exp(-cn^{1/3})$. In what follows, we consider $\omega$ in this set. Note that the second condition ensures $\omega_j \not\in [-Mn,Mn]$ for all $j\not\in[-2Mn, 2Mn]$ and thus these indices may be dropped from consideration. We first estimate
\begin{equation}
  \label{eq:mu(f)}
\begin{split}
  &\left|\int_{a_j^-/n}^{a_i^+/n} f_\lambda(x)\tilde{\mu}_n^{\alpha,\lambda,\omega}(\dd x)
  -\int_{a_j^-/n}^{a_i^+/n} f_\lambda(x)\mu_n^{\alpha,0,\omega}(\dd x)\right|\\
  &\quad \le \sum_{k:\:\omega_k\in[-Mn,Mn]}\frac{1}{n}|f_\lambda(\omega_k/n)|\sum_{l\in\Z\setminus\{k\}}\left(e^{\lambda|\omega_l-\omega_k|/n}-1\right)e^{-|\omega_l-\omega_k|^\alpha}.
\end{split}
\end{equation}
We divide the inner sum into two pieces and bound these separately. First, applying \eqref{eq:MDP},
\[\sum_{l\colon 0<|\omega_l-\omega_k|<n^{1/4}}\left(e^{\lambda|\omega_l-\omega_k|/n}-1\right)e^{-|\omega_l-\omega_k|^\alpha}
  \le c n^{2/3} \left(e^{\lambda n^{-3/4}}-1\right) \le c n^{-1/12}.\]
Second, if $\omega_k\in[-Mn,Mn]$ and $l\not\in[-2Mn, 2Mn]$, then on \eqref{eq:MDP} it holds that
\[\left|\omega_l-\omega_k\right|\geq  \frac{3|l|}{4}-|\omega_k|\geq \frac{|l|}{4},\]
and so
\begin{align*}
\sum_{l\colon |\omega_l-\omega_k|\ge n^{1/4}}\left(e^{\lambda|\omega_l-\omega_k|/n}-1\right)e^{-|\omega_l-\omega_k|^\alpha}
&\le \sum_{l\colon |\omega_l-\omega_k|\ge n^{1/4}}e^{-(1-\lambda/n)|\omega_l-\omega_k|}\\
&\le Cn e^{-(1-\lambda/n)n^{1/4}}+2\sum_{l\geq 2Mn}e^{-(1-\lambda/n)l/4}\\
&\leq C e^{-cn^{1/4}}.
\end{align*}
Substituting these bounds into \eqref{eq:mu(f)} and noting that \eqref{eq:MDP} implies $\#\{k\colon -Mn \le \omega_k \le Mn\} \le 4M n$ for sufficiently large $n$, we obtain
\begin{align}
  \label{eq:inv_meas1st}
  \left|\int_{a_j^-/n}^{a_i^+/n} f_\lambda(x)\tilde{\mu}_n^{\alpha,\lambda,\omega}(\dd x)
  -\int_{a_j^-/n}^{a_i^+/n} f_\lambda(x)\mu_n^{\alpha,0,\omega}(\dd x)\right| \le c n^{-1/12}.
\end{align}

For the next step, we introduce $\osc(f;I) := \sup_{x,y\in I}|f(x)-f(y)|$. By the uniform continuity of $f_\lambda\in C_b[-M,M]$, we can make $\max\{\osc(f_\lambda;[\tfrac{l}{N}, \tfrac{l+1}{N}])\colon -MN\le l \le MN-1\}$ as small as we wish by setting $N$ large. Let
\[\mathcal L:=\left\{l\in\Z\colon [\tfrac{l}{N}, \tfrac{l+1}{N}]\subseteq [a_j^-/n,a_i^+/n]\right\},\quad \partial \mathcal L:=\left\{l\in\Z\colon [\tfrac{l}{N}, \tfrac{l+1}{N}]\cap \{a_j^-/n,a_i^+/n\}\neq \emptyset\right\}.\]
With this notation, we obtain
\begin{align*}
  \lefteqn{\left|\int_{a_j^-/n}^{a_i^+/n}f_\lambda(x)\mu_n^{\alpha,0,\omega}(\dd x)-
  \int_{a_j^-/n}^{a_i^+/n}f_\lambda(x)\mathbf{E}[c^{\alpha,0}(\omega_0)]\dd x\right|}\\
  &\le \sum_{l\in\mathcal L} \osc(f_\lambda;[\tfrac{l}{N}, \tfrac{l+1}{N}])\left(\mu_n^{\alpha,0,\omega}([\tfrac{l}{N}, \tfrac{l+1}{N}])+\mathbf{E}[c^{\alpha,0}(\omega_0)]\tfrac{1}{N}\right)\\
  &\quad+\sum_{l\in\mathcal L} \|f_\lambda\|_\infty\left|\mu_n^{\alpha,0,\omega}([\tfrac{l}{N}, \tfrac{l+1}{N}])-\mathbf{E}[c^{\alpha,0}(\omega_0)]\tfrac{1}{N}\right|\\
  &\quad+ \sum_{l\in\partial\mathcal L} \|f_\lambda\|_\infty\left(\mu_n^{\alpha,0,\omega}([\tfrac{l}{N}, \tfrac{l+1}{N}])+\mathbf{E}[c^{\alpha,0}(\omega_0)]\tfrac{1}{N}\right).
\end{align*}
For any $N\in\N$, Lemma~\ref{lem:quantitativeLLN} ensures that
\begin{align*}
  \max\left\{\left|\mu_n^{\alpha,0,\omega}([\tfrac{l}{N}, \tfrac{l+1}{N}])-\mathbf{E}[c^{\alpha,0}(\omega_0)]\tfrac{1}{N}\right| \colon -MN \le l \le MN-1 \right\}\xrightarrow{n\to\infty} 0
\end{align*}
in $\mathbf{P}$-probability. Substituting this into the preceding bound, we obtain
\begin{align}
  \label{eq:inv_meas2nd}
  \int_{a_j^-/n}^{a_i^+/n}f_\lambda(x)\mu_n^{\alpha,0,\omega}(\dd x)-
  \int_{a_j^-/n}^{a_i^+/n}f_\lambda(x)\mathbf{E}[c^{\alpha,0}(\omega_0)]\dd x
  \xrightarrow{n\to\infty}0
\end{align}
in $\mathbf{P}$-probability. Combining \eqref{eq:inv_meas1st} and \eqref{eq:inv_meas2nd} yields \eqref{eq:meas_goal}, as desired.
\end{proof}

\section{Barrier crossing time estimates}\label{sec:rw}

In this section, we study the time taken by the Mott random walk to cross resistance barriers. One technical complication compared to the sketch provided in Section~\ref{sec:intro} is that we need to take the non-nearest-neighbour edges into account. In this general setting, the effective resistance  between two sets $A,B\subseteq\omega$ is defined by
\begin{align}\label{eq:def_Reff}
R_\mathrm{eff}(A,B)^{-1}:=\inf\left\{\mathcal E(f,f)\:\vline\:f\colon\omega\to[0,1],f|_A\equiv 0,f|_B\equiv 1 \right\},
\end{align}
where the energy of a function $f\colon\omega\to[0,1]$ is given by
\begin{align}\label{eq:def_E}
\mathcal E(f,f):=\sum_{i\neq j\in\Z}c(\omega_i,\omega_j)(f(\omega_i)-f(\omega_j))^2.
\end{align}
We also abbreviate $R(a,b):=R(\{a\},\{b\})$ in the case where $A$ and $B$ are singletons.

The constants $\alpha$, $\lambda$, $\delta$  and $K$ will be fixed throughout, and we recall the event $A_{n}^{\delta,K}$ from Definition~\ref{def:A}. Moreover, we recall the notation for barrier locations of $(a_i^+)_{i\geq 0}$ from \eqref{adef}, and also introduce
\[b_i^+:=a_i^++1,\qquad i\geq 0,\]
so that $\{\omega_{a_i^+},\omega_{b_i^+}\}$, $i\geq0$, represent edges of record resistance in the environment. We define corresponding exceedance times by setting, for $i\geq 0$,
\begin{align}
\alpha_i^+&:=\inf\left\{t\geq 0:X_t\geq \omega_{a_i^+}\right\},\label{eq:alphadef}\\
\beta_i^+&:=\inf\left\{t\geq 0:X_t\geq \omega_{b_i^+}\right\}.\label{eq:betadef}
\end{align}
The main result of this section is the following, which shows that, up to powers of $\ell_1(n)$, these exceedance times concentrate on environment-measurable quantities. Note that, from the definition at \eqref{gdef}, we have $g_i^+:=r^0(\omega_{a_i^+},\omega_{b_i^+})$.

\begin{theorem}\label{thm:annealed}
For each $\alpha>1$, $\lambda\in\mathbb{R}$, $K\in \mathbb{N}$ and $\delta>0$, as $n\rightarrow\infty$,
\begin{align}
\sup_{\omega\in A_n^{\delta,K}}\sup_{k=1,\dots, K} P_\omega^{\alpha,\lambda/n}\left(\alpha_k^+\geq (g_{k-1}^++\gtt{k}) n\ell_1(n)^4\right)\rightarrow0,\label{eq:not_too_fast}\\
\sup_{\omega\in A_n^{\delta,K}}\sup_{k=1,\dots, K} P_\omega^{\alpha,\lambda/n}\left(\beta_k^+\leq \frac{ng_k^+}{\ell_1(n)^{16}}\right)\rightarrow0.\label{eq:not_too_slow}
\end{align}
\end{theorem}

\subsection{An auxiliary random walk}\label{sec31}

Towards proving Theorem~\ref{thm:annealed}, it will be helpful to introduce, for each $n$ and $k$, a finite graph $\mathcal{G}_k^+\subseteq\omega$. Intuitively, we want the behavior of the Mott random walk on $\mathcal G_k^+$ to be the same as on $\omega$ until at least time $\beta_k^+$. We therefore include a sufficiently large interval to the left of the origin, which ensures that the random walk on $\omega$ leaves $\mathcal G_k^+$ by crossing $\{\omega_{a_k^+},\omega_{b_k^+}\}$ with high probability.

To define this, first recall the definition of $\at{k}$ from \eqref{eq:def_ex} and write, similarly to the corresponding definition on the positive axis, $\bt{k}=\at{k}-1$, so that $\{\omega_{\bt{k}},\omega_{\at{k}}\}$ is the closest nearest-neighbour edge to the left-hand side of the origin whose resistance exceeds $g_k^+$. We then define the vertex set of $\mathcal{G}_k^+$ to be
\[V\left(\mathcal{G}_k^+\right):=\left\{\omega_{\bt{k}},\omega_{\at{k}},\dots,\omega_{a_k^+},\omega_{b_k^+}\right\},\]
and the edge set to be all possible (unordered) pairs of distinct vertices in $V(\mathcal{G}_k^+)$. We will consider this as an electrical network with (symmetric) conductances given by
\begin{align*}
c^{\alpha,\lambda/n}_k(\omega_i,\omega_j):=\left\{\begin{array}{ll}
  c^{\alpha,\lambda/n}(\omega_i,\omega_j), & \mbox{if }\at{k}\leq i,j\leq a_k^+,\\
  \sum_{l\geq b_k^+}c^{\alpha,\lambda/n}(\omega_i,\omega_l), & \mbox{if }\at{k}\leq i\leq a_k^+,\:j=b_k^+,\\
  \sum_{l\leq \bt{k}}c^{\alpha,\lambda/n}(\omega_l,\omega_j), & \mbox{if }i=\bt{k},\:\at{k}\leq j\leq a_k^+,\\
  \sum_{l\leq \bt{k}}\sum_{m\geq b_k^+}c^{\alpha,\lambda/n}(\omega_l,\omega_m),  &  \mbox{if }i=\bt{k},\:j=b_k^+.
\end{array}\right.
\end{align*}
In particular, this is the network obtained from the original network defining the Mott random walk by identifying all the vertices $\{b_k^+,b_k^++1,\dots\}$ with $b_k^+$ and all the vertices $\{\dots,\bt{k}-1,\bt{k}\}$ with $\bt{k}$. See Figure~\ref{fig:aux_graph} for an example realisation of the configuration of vertices within $\mathcal{G}_k^+$ with $k=2$ and $3$. We note that, by basic properties of Poisson processes, it is possible to check that, $\mathbf{P}$-a.s., the graph $\mathcal{G}_k^+$ is well-defined and the conductances described above are all finite. Moreover, if we consider $(P_k,X^k)$ to be the continuous time Markov chain started from 0 with generator $L_k^{\alpha,\lambda/n}$, as characterised by
\begin{align*}
(L_k^{\alpha,\lambda/n}f)(\omega_i)&:=\sum_{j\in \{\bt{k},...,b_k^+\}}\frac{c_k^{\alpha,\lambda/n}(\omega_i,\omega_j)}{c_k^{\alpha,\lambda/n}(\omega_i)}\left(f(\omega_j)-f(\omega_i)\right),
\end{align*}
where we set
\[c_k^{\alpha,\lambda/n}(\omega_i):=\left\{\begin{array}{ll}
                                                    \sum_{j\in \{{\bt{k}},{\at{k}},\dots,{a_k^+},{b_k^+}\}\backslash\{i\}}c_k^{\alpha,\lambda/n}(\omega_i,\omega_j), & \mbox{if }\at{k}\leq i\leq a_k^+,\\
                                                    1, & \mbox{if }i\in\{\bt{k},\:b_k^+\},
                                                  \end{array}\right.\]
then $X^k$ behaves exactly as the Mott random walk $X$, up until the time it exits the set $\{\omega_{\at{k}},\dots,\omega_{a_k^+}\}$. Note that the choice of $c_k^{\alpha,\lambda/n}(\omega_i)=1$ for $i\in\{\bt{k},\:b_k^+\}$ is arbitrary for this subsection, but will be convenient for the argument of the next subsection.

\begin{figure}[ht!]
\begin{center}
\includegraphics[width=.9\textwidth]{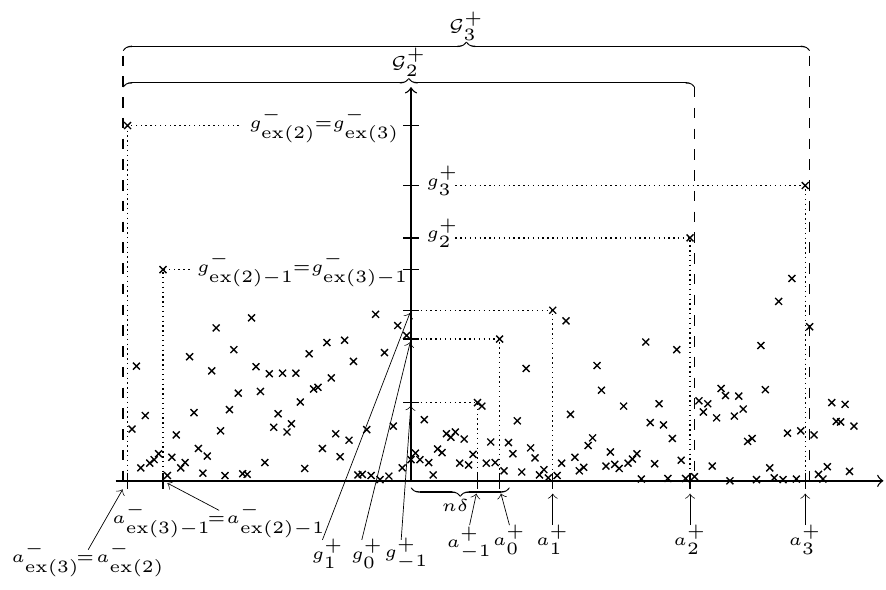}
\end{center}
\caption{An illustration of the point process $(i,r^{\alpha,0}(\omega_i,\omega_{i+1}))_{i\in\Z}$ together with the indices used to define the auxiliary graphs $\mathcal G_2^+$ and $\mathcal G_3^+$. \label{fig:aux_graph}}
\end{figure}

Let $R^k_{\mathrm{eff}}$ denote the effective resistance on $\mathcal{G}_k^+$, which is defined similarly to \eqref{eq:def_Reff}. We close this subsection by proving upper and lower bounds for $R^k_{\mathrm{eff}}$. We recall that, by the definition in Section~\ref{sec:prelim},
\[\gtt{k}=\max_{j\in\{\at{k},\dots,-1\}}r^{\alpha,0}(\omega_j,\omega_{j+1}),\]
which represents the size of the largest barrier between $\omega_{\at{k}}$ and the origin. (Again, see Figure~\ref{fig:aux_graph}.)

\begin{lemma}\label{lem:R_bounds}
There exists $C>1$ such that the following hold on $A_n^{\delta,K}$, for all $n$ large enough.
\begin{enumerate}
 \item[(i)] For all $A,B\subseteq\mathcal G_k^+$ with $A\cap \{\omega_{\at{k}},...,\omega_{a_k^+}\}\neq\emptyset$, $B\cap \{\omega_{\at{k}},...,\omega_{a_k^+}\}\neq\emptyset$,
\begin{align*}
{R}^k_\mathrm{eff}(A,B)\leq C (g_{k-1}^++\gtt{k})\ell_1(n).
\end{align*}
\item[(ii)] For all $A\subseteq\{\omega_{\at{k}},...,\omega_{a_k^+}\}$ and $B\subseteq \{\omega_{\bt{k}},\omega_{b_k^+}\}$,
\begin{align*}
{R}^k_\mathrm{eff}(A,B)\geq \frac{g_k^+}{C\ell_1(n)^5}.
\end{align*}
\item[(iii)] ${R}^k_\mathrm{eff}(\at{k},\omega_{a_k^+})\geq \frac{g_{k-1}^++\gtt{k}}{C\ell_1(n)^5}$.
\end{enumerate}
\end{lemma}

\begin{proof}
We start with part (i). By Rayleigh's monotonicity property (see \cite[Section 1.4.1]{DS}, for example), we may assume that $A=\{\omega_i\}$ and $B=\{\omega_j\}$ for some $\at{k}\leq i<j\leq a_k^+$. Then
\begin{eqnarray}
{R}^k_\mathrm{eff}(\omega_i,\omega_j)&\leq& {\textstyle \sum_{l=i+1}^{j}}{R}^k_\mathrm{eff}(\omega_{l-1},\omega_{l})\nonumber\\
&\leq &{\textstyle \sum_{l=\at{k}}^{a_k^+}}{R}^k_\mathrm{eff}(\omega_{l-1},\omega_{l})\nonumber\\
&\leq & {\textstyle \sum_{l=\at{k}}^{a_k^+}}r^{\alpha,\lambda/n}(\omega_{l-1},\omega_l)\nonumber\\
&=&{\textstyle \sum_{l=\at{k}}^{a_k^+}}r^{\alpha,0}(\omega_{l-1},\omega_l)e^{-\lambda(\omega_{l-1}+\omega_l)/n},\nonumber\\
&\leq & C(g_{k-1}^++\gtt{k})e^{2|\lambda| \ell_3(n)}.\nonumber
\end{eqnarray}
We have used the triangle inequality for the resistance metric (see \cite[Theorem 2.64]{Barbook}, for example) in the first inequality. The third inequality is again due to Rayleigh's monotonicity property and the final inequality follows from \eqref{eq:space_scaling} and \eqref{eq:NN+}--\eqref{eq:NN-}, which hold on $A_n^{\delta,K}$. The claim now follows from $\ell_3=\log\log \ell_1$.

For part (ii), using again the monotonicity property of the effective resistance and the parallel law (see \cite[Remark 2.52]{Barbook}, for example), we have that
\begin{equation}\label{eq:Reff0bb}
\begin{split}
{R}^k_\mathrm{eff}(A,B)&\geq {{R}^k_\mathrm{eff}\left(V\left(\mathcal{G}_k^+\right)\backslash\{\omega_{\bt{k}},\omega_{{b}_k^+}\},\{\omega_{\bt{k}},\omega_{{b}_k^+}\}\right)}\\
&\geq \left( {\textstyle \sum_{m=\at{k}}^{a_k^+}}c^{\alpha,\lambda/n}_k(\omega_{\bt{k}},\omega_m)+ {\textstyle \sum_{l=\at{k}}^{a_k^+}}c^{\alpha,\lambda/n}_k(\omega_l,\omega_{b_k^+})\right)^{-1}\\
&=\left( {\textstyle \sum_{l\leq \bt{k}}\sum_{m\geq \at{k}}}c^{\alpha,\lambda/n}(\omega_l,\omega_m)+ {\textstyle \sum_{l\leq a_k^+}\sum_{m\geq b_k^+}}c^{\alpha,\lambda/n}(\omega_l,\omega_m)\right)^{-1}.
\end{split}
\end{equation}
Now, note that, since $\alpha>1$, for $l\leq\bt{k}$, $m\geq \at{k}$,
\begin{eqnarray*}
\lefteqn{c^{\alpha,\lambda/n}(\omega_l,\omega_m)}\\
&=&e^{-|\omega_l-\omega_m|^\alpha+\lambda(\omega_l+\omega_m)/n}\\
&\leq&e^{-|\omega_l-\omega_{\bt{k}}|^\alpha-|\omega_{\bt{k}}-\omega_{\at{k}}|^\alpha-|\omega_{\at{k}}-\omega_m|^\alpha+\lambda(\omega_l+\omega_m)/n}\\
&=&c^{\alpha,\lambda/n}(\omega_l,\omega_{\bt{k}})c^{\alpha,0}(\omega_{\bt{k}},\omega_{\at{k}})c^{\alpha,\lambda/n}(\omega_{\at{k}},\omega_m)e^{-2\lambda(\omega_{\bt{k}}+\omega_{\at{k}})/n}.
\end{eqnarray*}
Hence, on $A_n^{\delta,K}$, for $n$ large enough,
\begin{equation}
  \label{eq:leftbridges}
\begin{split}
\lefteqn{  {\textstyle \sum_{l\leq\bt{k}}\sum_{m\geq \at{k}}}c^{\alpha,\lambda/n}(\omega_l,\omega_m)}\\&\leq c^{\alpha,\lambda/n}(\omega_{\bt{k}})c^{\alpha,0}(\omega_{\bt{k}},\omega_{\at{k}})c^{\alpha,\lambda/n}(\omega_{\at{k}})e^{-2\lambda(\omega_{\bt{k}}+\omega_{\at{k}})/n}\\
  &\leq C(g_k^+)^{-1}\ell_1(n)^4{\ell_2(n)^{2\lambda}}\\
  &\leq C(g_k^+)^{-1}\ell_1(n)^5,
\end{split}
\end{equation}
where for the second inequality we have applied \eqref{eq:space_scaling} and \eqref{eq:pointwise} and the fact that
\[c^{\alpha,0}(\omega_{\bt{k}},\omega_{\at{k}})=r^{\alpha,0}(\omega_{\bt{k}},\omega_{\at{k}})^{-1}\leq r^{\alpha,0}(\omega_{{a}_k^+},\omega_{{b}_k^+})^{-1}=(g_k^+)^{-1}.\]
Applying the same argument to the second sum in \eqref{eq:Reff0bb}, we obtain
\[R^{k}_{\mathrm{eff}}\left(A,B\right)\ge \left((g_{k}^+)^{-1}+(\gt{k})^{-1}\right)^{-1}/{C\ell_1(n)^5}\ge  \frac{g_{k}^+}{2C\ell_1(n)^5}.\]
In the final inequality, we have again used that $g_k^+\leq \gt{k}$. We turn to the proof of part (iii). Similar to \eqref{eq:Reff0bb}--\eqref{eq:leftbridges}, one can show that
\[R^k_{\mathrm{eff}}\left(\omega_{\at{k}},{\omega_{a_k^+}}\right)\geq R^k_{\mathrm{eff}}\left(\{\omega_{\bt{k}},\dots,\omega_{{a}_{k-1}^+}\},\{\omega_{b_{k-1}^+},\omega_{b_k^+}\}\right)\geq \frac{g_{k-1}^+}{C\ell_1(n)^5}.\]
Since the same argument applies when we consider the edge of nearest neighbour resistance $\gtt{k}$, it follows that
\begin{align*}
R^k_{\mathrm{eff}}\left(\omega_{\at{k}},{\omega_{a_k^+}}\right)\geq C \max\{\gtt{k},g_{k-1}^+\}/\ell_1(n)^5\geq  \frac C2 (\gtt{k}+g_{k-1}^+)/ \ell_1(n)^5.\tag*{\qedhere}
\end{align*}
\end{proof}

\subsection{Upper bound on barrier crossing times}

Using definitions analogous to \eqref{eq:alphadef} and \eqref{eq:betadef} for exceedance times on the negative axis, we have that
\begin{align*}
\betat{k}=\inf\{t\geq 0:X_t\leq \omega_{\bt{k}}\}
\end{align*}
is the first time the random walk on $\omega$ leaves $\mathcal G_k^+$ on the left-hand side. As a first step towards proving \eqref{eq:not_too_fast}, we check that, on the event $A_n^{\delta,K}$, there is a high quenched probability that the Mott random walk $X$ exceeds $\omega_{a_k^+}$ before time $\betat{k}$.

\begin{lemma}\label{lem:l1} There exists a $C\in (0,\infty)$ such that for $n\in\mathbb{N}$ and $1\leq k\leq K$, on $A_n^{\delta,K}$,
\[P_\omega^{\alpha,\lambda/n}\left(\alpha_k^+\geq \betat{k}\right)\leq C\ell_1(n)^6e^{-(\log n)^{\alpha-1}/\ell_2(n)}.\]
\end{lemma}
\begin{proof} For a set $A\subseteq \{\bt{k},...,b_k^+\}$, write
\begin{align}\label{eq:defT}
T_A:=\inf\{t\geq 0:\:X^k_t=\omega_i\text{ for some }i\in A\}
\end{align}
for the hitting time of $A$ by $X^k$. By construction, we have that
\[P_\omega^{\alpha,\lambda/n}\left(\alpha_k^+\geq \betat{k}\right)=\Pk\left(T_{\{a_k^+,b_k^+\}}\geq T_{\bt{k}}\right).\]
Here and in the following, we write $T_i$ instead of $T_{\{i\}}$ for simplicity. On $A_n^{K,\delta}$, we can bound the latter expression as follows:
\begin{align}
\Pk\left(T_{\{a_k^+,b_k^+\}}\geq T_{\bt{k}}\right)&\leq
\frac{{R}^k_\mathrm{eff}(0,\{\omega_{a_k^+},\omega_{b_k^+}\})}{{R}^k_\mathrm{eff}(0,\omega_{\bt{k}})}\label{eq:bbb1}\\
&\leq \frac{C(g_{k-1}^++\gtt{k})\ell_1(n)^6}{g_{k}^+}\nonumber\\
&\leq C\ell_1(n)^6e^{-(\log n)^{\alpha-1}/\ell_2(n)},\nonumber
\end{align}
where the first inequality is due to \cite[Lemma 2.62]{Barbook}, in the second inequality we have applied the upper and lower bounds from Lemma~\ref{lem:R_bounds}(i)--(ii) and the final inequality is due to  \eqref{eq:separation}.
\end{proof}

It is thus enough to estimate the time taken by $X$ to leave the set $\{\omega_{\at{k}},\dots,\omega_{a_k^+-1}\}$, which is what we do next.

\begin{lemma}\label{l2} There exists a $C\in (0,\infty)$ such that for $n\in\mathbb{N}$ and $1\leq k\leq K$, on $A_n^{\delta,K}$,
\[P_\omega^{\alpha,\lambda/n}\left(\min\{\alpha_k^+,\betat{k}\}\geq (g_{k-1}^++\gtt{k})n\ell_1(n)^4\right)\leq \frac{C}{\ell_1(n)}.\]
\end{lemma}
\begin{proof} It is a straightforward consequence of the commute time identity (see \cite{CRRST} or \cite[Theorem~2.63]{Barbook}) that
\[
E_k\left(T_{\{\bt{k},a_k^+,b_k^+\}}\right)\leq R^k_{\mathrm{eff}}\left(0,\{\omega_{\bt{k}},\omega_{a_k^+},\omega_{b_k^+}\}\right){\textstyle \sum_{l=\at{k}}^{a_k^+-1}}c_k^{\alpha,\lambda/n}(\omega_l),\]
where we again write $T_A$ for the hitting time of a set $A$ by $X^k$, see \eqref{eq:defT}. By \eqref{eq:mass_control}, it holds that
\[{\textstyle \sum_{l=\at{k}}^{a_k^+-1}}c_k^{\alpha,\lambda/n}(\omega_l)={\textstyle \sum_{l=\at{k}}^{a_k^+-1}}c^{\alpha,\lambda/n}(\omega_l)\leq Cn\ell_1(n)^2.\]
Hence, using Lemma~\ref{lem:R_bounds}(i) and  the equivalence of the laws of $X$ and $X^k$ up to the exit time of $\{\omega_{\at{k}},\dots,\omega_{a_k^+}\}$, we obtain that
\[E_\omega^{\alpha,\lambda/n}\left(\min\{\alpha_k^+,\betat{k}\}\right)\leq  C(g_{k-1}^++\gtt{k})n\ell_1(n)^3.\]
The result follows by applying this bound in conjunction with Markov's inequality.
\end{proof}

\begin{proof}[Proof of \eqref{eq:not_too_fast}] A simple union bound gives
\begin{eqnarray*}
\lefteqn{P_\omega^{\alpha,\lambda/n}\left(\alpha_k^+\geq (g_{k-1}^++\gtt{k})n\ell_1(n)^4\right)}\\
&\leq& P_\omega^{\alpha,\lambda/n}\left(\min\{\alpha_k^+,\betat{k}\}\geq (g_{k-1}^++\gtt{k})n\ell_1(n)^4\right) +
 P_\omega^{\alpha,\lambda/n}\left(\alpha_k^+\geq \betat{k}\right).
\end{eqnarray*}
Hence, combining Lemmas~\ref{lem:l1} and \ref{l2} gives the result.
\end{proof}

\subsection{Lower bound on barrier crossing times}\label{sec33}

For proving \eqref{eq:not_too_slow}, it will be convenient to continue working with the graph $\mathcal{G}_k^+$ defined at the beginning of Subsection~\ref{sec31}. For a given realisation of such a graph, we introduce stopping times for the random walk $X^k$ by setting $\sigma(0):=0$ and, for $i\geq 0$,
\begin{align*}
\tau(i)&:=\inf\left\{t\in[\sigma(i),\Delta]:\:X(t)= \omega_{a_k^+}\right\},\\
\sigma(i+1)&:=\inf\left\{t\in [\tau(i),\Delta]:\:X(t)=\omega_{\at{k}}\right\},
\end{align*}
where $\Delta:=T_{\{\bt{k},b_k^+\}}$ (we again write $T_A$ for the hitting time of a set $A$ by $X^k$, see \eqref{eq:defT}), and we set the infimum of an empty set to be $\infty$ by convention. Furthermore, let $G:=\inf\left\{i\geq 0:\:\tau(i)=\infty\right\}$. Clearly, if $G\geq 2$, then $\Delta\geq S_{G-1}$, where $S_m:=\sum_{i=1}^{m}(\tau_n(i)-\tau_n(i-1))$. Since $\Delta$ is equal in distribution to $\min\{\betat{k},\beta_k\}$, we have
\begin{align}
  \label{eq:betatoS}
  P_\omega^{\alpha,\lambda/n}\left(\beta_k^+\geq \frac{ng_k^+}{\ell_1(n)^{16}}\right)
  \geq \Pk\left(S_{G-1}\geq ng_k^+/\ell_1(n)^{16}\right).
\end{align}
It will thus suffice to understand the distribution of $S_{G-1}$. The next two lemmas give estimates on $G$ and the excursion time $\tau(1)-\tau(0)$, respectively.  Combining these will give the result of interest.

\begin{lemma}\label{lem:geom} There exists $C_1\in (0,\infty)$ such that for $n\in\mathbb{N}$ and $1\leq k\leq K$, on $A_n^{\delta,K}$,
\[\Pk\left(G\geq \frac{ g_{k}^+}{C_1 \ell_1(n)^7(\gtt{k}+g_{k-1}^+)}\right)\geq e^{-C_1\ell_1(n)^{-1}}.\]
\end{lemma}
\begin{proof} First note that, similarly to \eqref{eq:bbb1},
\[\Pk(\tau(0)=\infty)=\Pk\left(T_{a_k^+}>\Delta\right)\leq
\frac{{R}^k_\mathrm{eff}(0,\omega_{a_k^+})}{{R}^k_\mathrm{eff}(0,\{\omega_{\bt{k}},\omega_{{b}_k^+}\})}.\]
By Lemma~\ref{lem:R_bounds}, on $A_n^{\delta,K}$, the numerator here can be bounded above by $C(g_{k-1}^++\gtt{k})\ell_1(n)$ while the denominator is bounded from below by $Cg_k^+/\ell_1(n)^5$. Hence we conclude that
\begin{equation}\label{taubound}
\Pk(\tau(0)=\infty)\leq \frac{C\ell_1(n)^6(g_{k-1}^++\gtt{k})}{g_{k}^+}.
\end{equation}
For subsequent stopping times, we have that
\[\Pk(\sigma(i+1)=\infty\:|\:\tau(i)<\infty)=\Pk
\left(T_{\at{k}}>\Delta\:\vline\:X^k_0=\omega_{a_k^+}\right)\leq
\frac{{R}^k_\mathrm{eff}(\omega_{\at{k}},\omega_{a_k^+})}{{R}^k_\mathrm{eff}(\omega_{a_k^+},\{\omega_{\bt{k}},\omega_{{b}_k^+}\})},\]
and also
\[\Pk(\tau(i)=\infty\:|\:\sigma(i)<\infty)=\Pk
\left(T_{{a}_k^+}>\Delta\:\vline\:X^k_0=\omega_{\at{k}}\right)\leq
\frac{{R}^k_\mathrm{eff}(\omega_{\at{k}},\omega_{a_k^+})}{{R}^k_\mathrm{eff}(\omega_{\at{k}},\{\omega_{\bt{k}},\omega_{{b}_k^+}\})}.\]
Again applying Lemma~\ref{lem:R_bounds}, together with a strong Markov property argument, gives that
\[\begin{split}
  \Pk(\tau(i+1)=\infty\:|\:\tau(i)<\infty)\leq \frac{C\ell_1(n)^6(\gtt{k}+g_{k-1}^+)}{{g_{k}^+}}
\end{split}\]
for appropriately adjusted $C$. Applying the above bounds, we find that
\[\Pk\left(G\geq i+1\right)=\Pk\left(\tau(i)<\infty\right)\geq \left(1- \frac{C\ell_1(n)^6(\gtt{k}+g_{k-1}^+)}{{g_{k}^+}}\right)^i.\]
Together with \eqref{eq:separation}, the result follows.
\end{proof}

Next, we give a lower bound on the excursion times.

\begin{lemma}\label{lem:exc}
There exists $C_2\in (0,\infty)$ such that for $n\in\mathbb{N}$ and $1\leq k\leq K$, on $A_n^{\delta,K}$,
\[\Pk\left(\tau(i)-\tau(i-1)\geq \frac{C_2(\gtt{k}+g_{k-1}^+)n}{ \ell_1(n)^{5}}\cond\tau(i-1)<\infty\right)\geq \frac{C_2}{\ell_1(n)^{8}}.\]
\end{lemma}
\begin{proof} Let
\[T_{\at{k}}^{a_k^+}:=\inf\left\{t\geq T_{\at{k}}:X^k_t=\omega_{a_k^+}\right\}\]
be the first time that $X^k$ hits $\omega_{a_k^+}$ after hitting $\omega_{\at{k}}$. It then holds that
\begin{eqnarray}
{\Pk\left(\tau(i)-\tau(i-1)\leq t\:|\:\tau(i-1)<\infty\right)}&=&\Pk\left(T_{\at{k}}^{a_k^+}\leq \min\left\{t,T_{\{\bt{k},b_k^+\}}\right\}\:\vline\:X^k_0=\omega_{a_k^+}\right)\nonumber\\
&\leq &\Pk\left(T_{\at{k}}^{a_k^+}\leq t\:\vline\:X^k_0=\omega_{a_k^+}\right).\label{qq0}
\end{eqnarray}
To obtain a lower bound for this probability, we consider the process $\tilde{X}^k=(\tilde{X}^k_t)_{t\geq 0}$ obtained by observing $X^k$ on the set
\[\tilde{V}_k:=V\left(\mathcal{G}_k^+\right)\backslash\left\{\omega_{\bt{k}},\omega_{b_k^+}\right\}=\left\{\omega_{\at{k}},\dots,\omega_{a_k^+}\right\}.\]
Precisely, define an additive functional $\tilde{A}^k=(\tilde{A}^k_t)_{t\geq 0}$ by setting $\tilde{A}^k_t:=\int_0^t\mathbf{1}_{\{X^k_s\in\tilde{V}_k\}}ds$, and then let
\begin{equation}\label{xtildekdef}
\tilde{X}^k_t:=X^k_{(\tilde{A}^k)^{-1}_t},
\end{equation}
where $(\tilde{A}^k)^{-1}$ is the right-continuous inverse of $\tilde{A}^k$. By standard theory concerning the traces of Dirichlet and resistance forms (see \cite[Theorem 6.2.1]{FOT}, as described in \cite[Theorem 4.17]{Barlow}, and \cite[Theorem 8.4]{Kigami}), $\tilde{X}^k$ is the continuous-time Markov chain on $\tilde{V}_k$ corresponding to the resistance metric $R^k_{\mathrm{eff}}|_{\tilde{V}_k\times \tilde{V}_k}$ and measure $c_k^{\alpha,\lambda/n}|_{\tilde{V}_k}$. Defining $\tilde{T}_{\at{k}}^{a_k^+}$ from $\tilde{X}^k$ analogously to the definition of $T_{\at{k}}^{a_k^+}$ from $X^k$, we have that
\[\tilde{T}_{\at{k}}^{a_k^+}\leq {T}_{\at{k}}^{a_k^+}.\]
(Indeed, the difference is given by the time spent by $X^k$ in the set $\{\omega_{\bt{k}},\omega_{b_k^+}\}$ before ${T}_{\at{k}}^{a_k^+}$.) Hence, \eqref{qq0} implies
\begin{equation}\label{qq01}
\Pk\left(\tau(i)-\tau(i-1)\leq t\:|\:\tau(i-1)<\infty\right)
\leq \Pk\left(\tilde{T}_{\at{k}}^{a_k^+}\leq t\:\vline\:\tilde{X}^k_0=\omega_{a_k^+}\right).
\end{equation}
In estimating the right-hand side here, we use the following sequence of inequalities
\begin{equation}\label{eq:sdads}
\begin{split}
& \Ek\left(\tilde{T}_{\at{k}}^{a_k^+}\:\vline\:\tilde{X}^k_0=\omega_{a_k^+}\right)\\
&\leq t+\Ek\left(\mathbf{1}{\{\tilde{T}_{\at{k}}>t\}}\left(t+
{\textstyle \sum_{l=\at{k}}^{a_k^+}}\mathbf{1}_{\{\tilde{X}^k_t=\omega_l\}}\Ek\left(\tilde{T}_{\at{k}}^{a_k^+}\:\vline\:\tilde{X}^k_0=\omega_l\right)\right)\:\vline\:\tilde{X}^k_0=\omega_{a_k^+}\right)\\
&\quad +\Ek\left(\mathbf{1}{\{\tilde{T}_{\at{k}}\leq t,\tilde{T}_{\at{k}}^{a_k^+}>t\}}\left(t+
{\textstyle \sum_{l=\at{k}}^{a_k^+}}\mathbf{1}_{\{\tilde{X}^k_t=\omega_l\}}\Ek\left(\tilde{T}_{a_k^+}\:\vline\:\tilde{X}^k_0=\omega_l\right)\right)\:\vline\:\tilde{X}^k_0=\omega_{a_k^+}\right)\\
&\leq t+\Ek\left(\mathbf{1}{\{\tilde{T}_{\at{k}}^{a_k^+}>t\}}\left(t+
{\textstyle \sum_{l=\at{k}}^{a_k^+}}\mathbf{1}_{\{\tilde{X}^k_t=\omega_l\}}\Ek\left(\tilde{T}_{\at{k}}^{a_k^+}\:\vline\:\tilde{X}^k_0=\omega_l\right)\right)\:\vline\:\tilde{X}^k_0=\omega_{a_k^+}\right),
\end{split}
\end{equation}
where we write $\tilde{T}_{{a_k^+}}$ for the hitting time of $\omega_{a_k^+}$ by $\tilde{X}^k$, and we have  used the Markov property in the first inequality and the fact that $\tilde{T}_{\at{k}}^{a_k^+}\geq \tilde{T}_{a_k^+}$ in the second inequality. To make use of \eqref{eq:sdads}, we will now derive a lower bound for the expectation in the first line and an upper bound for the inner expectations in the final line.

Towards the second goal, we observe that by the commute time identity (again, see \cite[Theorem 2.63]{Barbook}, \eqref{eq:mass_control} and Lemma~\ref{lem:R_bounds}(i),
for any $j\in \{\at{k},\dots,a_k^+\}$,
\begin{eqnarray}
\lefteqn{\Ek\left(\tilde{T}_{{\at{k}}}^{{a_k^+}}\:\vline\:\tilde{X}^k_0=\omega_{j}\right)}\nonumber\\
&\leq &\Ek\left(\tilde{T}_{{\at{k}}}^{j}\:\vline\:\tilde{X}^k_0=\omega_{j}\right)+\Ek\left(\tilde{T}_{{\at{k}}}^{{a_k^+}}\:\vline\:\tilde{X}^k_0=\omega_{a_k^+}\right)\nonumber\\
&\leq& R^k_{\mathrm{eff}}\left(\omega_{j},{\omega_{\at{k}}}\right){\textstyle \sum_{l=\at{k}}^{a_k^+}}c^{\alpha,\lambda/n}_k(\omega_l) + R^k_{\mathrm{eff}}\left(\omega_{\at{k}},{\omega_{a_k^+}}\right){\textstyle \sum_{l=\at{k}}^{a_k^+}}c^{\alpha,\lambda/n}_k(\omega_l)\nonumber\\
&\leq & C(\gtt{k}+g_{k-1}^+)n\ell_1(n)^3,\label{qq1}
\end{eqnarray}
where $\tilde T_{{\at{k}}}^{j}$ is defined analogously to $\tilde T_{{\at{k}} }^{{a_k^+}}$. Turning to the lower bound for the first line in \eqref{eq:sdads}, again applying the commute time identity and Lemma~\ref{lem:R_bounds}(iii), yields
\begin{align}
\Ek\left(\tilde{T}_{{\at{k}}}^{{a_k^+}}\:\vline\:\tilde{X}^k_0=\omega_{a_k^+}\right)
&=R^k_{\mathrm{eff}}\left(\omega_{\at{k}},{\omega_{a_k^+}}\right){\textstyle \sum_{l=\at{k}}^{a_k^+}}c^{\alpha,\lambda/n}_k(\omega_l)\nonumber\\
&\geq \frac{g_{k-1}^++\gtt{k}}{C\ell_1(n)^5}{\textstyle  \sum_{l=\at{k}}^{a_k^+}}c^{\alpha,\lambda/n}_k(\omega_l).\label{eq:ct}
\end{align}
As we also have from \eqref{eq:mass_control} and the definition $c_k^{\alpha,\lambda/n}(\omega_i)=1$ for $i\in\{\bt{k},b_k^+\}$ that
\[{\textstyle \sum_{l=\at{k}}^{a_k^+}}c^{\alpha,\lambda/n}_k(\omega_l)={\textstyle \sum_{l=\bt{k}}^{b_k^+}}c^{\alpha,\lambda/n}_k(\omega_l)-2 \geq Cn,\]
we obtain by inserting this bound into the inequality at \eqref{eq:ct} that
\begin{equation}\label{qq3}
\Ek\left(\tilde{T}_{{\at{k}}}^{{a_k^+}}\:\vline\:\tilde{X}^k_0=\omega_{a_k^+}\right)
\geq C(\gtt{k}+g_{k-1}^+)n/\ell_1(n)^5.
\end{equation}
We now have all the estimates we need to continue with the proof. Namely, inserting \eqref{qq1} and \eqref{qq3} into \eqref{eq:sdads} gives
\begin{align*}
\frac{C(\gtt{k}+g_{k-1}^+)n}{\ell_1(n)^5}&\leq  t+\Pk\left(\tilde{T}_{\omega_{\at{k}}}^{\omega_{a_k^+}}>t\:\vline\:\tilde{X}^k_0=\omega_{a_k^+}\right)\left(t+C(\gtt{k}+g_{k-1}^+)n\ell_1(n)^3\right).
\end{align*}
Rearranging this inequality gives that for some $C_3, C_4>0$,
\[\Pk\left(\tilde{T}_{{\at{k}}}^{{a_k^+}}>t\:\vline\:\tilde{X}^k_0=\omega_{a_k^+}\right)\geq
\frac{C_3(\gtt{k}+g_{k-1}^+)n \ell_1(n)^{-5}-t}{t+C_4(\gtt{k}+g_{k-1}^+)n\ell_1(n)^3}.\]
Consequently, taking $t=\tfrac{1}{2}C_3(\gtt{k}+g_{k-1}^+)n \ell_1(n)^{-5}$, we obtain
\[\Pk\left(\tilde{T}_{{\at{k}}}^{{a_k^+}}>\tfrac{1}{2}C_3(\gtt{k}+g_{k-1}^+)n \ell_1(n)^{-5}\:\vline\:\tilde{X}^k_0=\omega_{a_k^+}\right)\geq \frac{C}{\ell_1(n)^{8}}.\]
In conjunction with the bound at \eqref{qq01}, the result follows.
\end{proof}

We can now prove the main result of this subsection.

\begin{proof}[Proof of \eqref{eq:not_too_slow}]
On $A_n^{\delta,K}$, by applying the previous two lemmas, we have that
\begin{align*}
&\Pk\left(S_{G-1}\geq ng_k^+/\ell_1(n)^{16}\right)\\
&\quad \geq  e^{-C_1\ell_1(n)^{-1}}\Pk\left(S_{G-1}\geq \frac{ng_k^+}{\ell_1(n)^{16}}\cond G\geq \frac{g_{k}^+ }{C_1 \ell_1(n)^7(\gtt{k}+g_{k-1}^+)}\right)\nonumber\\
&\quad \geq e^{-C_1\ell_1(n)^{-1}}\Pk\left(B\geq \frac{C_1g_k^+}{C_2(\gtt{k}+g_{k-1}^+)\ell_1(n)^{16}}\right),
\end{align*}
where $C_1$ and $C_2$ are the constants of Lemmas~\ref{lem:geom} and \ref{lem:exc}, respectively, and $B$ is a binomial random variable with $\lfloor g_{k}^+ /C_1 \ell_1(n)^7(\gtt{k}+g_{k-1}^+)\rfloor-1$ trials and success probability ${C_2}/{\ell_1(n)^{8}}$. In particular, recalling that, on $A_n^{\delta,K}$, it holds that ${g_{k}^+}/(\gtt{k}+g_{k-1}^+)\geq e^{(\log n)^{\alpha-1}/\ell_2(n)}$, for suitably large $n$ the expectation and variance of $B$ are both contained in the interval
\[\left[\frac 12\frac{C_2 g_{k}^+}{C_1(\gtt{k}+g_{k-1}^+) \ell_1(n)^{15}},\frac 32\frac{C_2 g_{k}^+}{C_1(\gtt{k}+g_{k-1}^+) \ell_1(n)^{15}}\right].\]
Thus, by the Chebyshev inequality and \eqref{eq:separation},
\begin{align*}
\Pk\left(B\geq \frac{C_1g_k^+}{C_2(\gtt{k}+g_{k-1}^+)\ell_1(n)^{16}}\right)&\geq 1-\Pk\left(|B-\Ek[B]|\geq \frac {\Ek[B]}2\right)\\
&\geq 1-\frac{4\text{Var}_k[B]}{\Ek[B]^2}\\
&\xrightarrow{n\to\infty}1.
\end{align*}
From this,
\begin{align}
  \label{eq:SonA}
\inf_{\omega\in A_n^{\delta,K}}  \Pk\left(S_{G-1}\geq ng_k^+/\ell_1(n)^{16}\right)\xrightarrow{n\to\infty}1,
\end{align}
and recalling~\eqref{eq:betatoS}, we complete the proof.
\end{proof}

\begin{rem} \label{rem:SonA}
Note that $S_{G-1}\geq ng_k^+/\ell_1(n)^{16}$ implies $T_{{a_k^+}}<\infty$ and depends only on the behaviour of the random walk after $T_{{a_k^+}}$. Therefore, \eqref{eq:SonA} remains valid for the random walk starting from $\omega_{a_k^+}$.
\end{rem}

\section{Proof of main results}\label{sec:proof}

\subsection{Exceedance time and extrema scaling}

In this subsection, we will prove Theorems~\ref{thm:mainnew1} and \ref{thm:mainnew2}. To begin with, we apply the barrier crossing time estimates of Section~\ref{sec:rw} to deduce the following one-sided version of Theorem~\ref{thm:mainnew1}(a).

\begin{proposition} \label{prop:mainnew1}
Fix $\alpha>1$ and $\lambda\in\mathbb{R}$. As $n\rightarrow\infty$,
\[d_{U}\left(\left(n^{-1}L\left(n^{-1}\Delta^+_{nx}\right)\right)_{x\geq 0},\left(m_{n,+}(x)\right)_{x\geq 0}\right)\rightarrow 0\]
in $\mathbb{P}^{\alpha,\lambda/n}$-probability.
\end{proposition}
\begin{proof} Fix $\delta>0$ and $K\in\mathbb{N}$, and define $(\alpha_k^+,\beta_k^+)_{k\geq 0}$ as at \eqref{eq:alphadef} and \eqref{eq:betadef}.
Note that if $nx\in [\omega_{a_k^+},\omega_{a_{k+1}^+})$ for some $k\geq 1$, then $\Delta_{nx}^+\in[\beta_k^+,\alpha_{k+1}^+]$. Hence, on the event
\[C_n^{\delta,K}:={\textstyle\bigcap_{k=1}^K}\left\{\alpha_k^+<ng_{k-1}^+\ell_1(n)^\A,\:\beta_k^+>ng_{k}^+\ell_1(n)^{-\B}\right\},\]
it holds that: if $nx\in [\omega_{a_k^+},\omega_{a_{k+1}^+})$ for some $k\in\{1,\dots,K-1\}$, then
\[n^{-1}L\left(g_{k}^+\ell_1(n)^{-\B}\right)\leq n^{-1}L\left(n^{-1}\Delta_{nx}^+\right)\leq n^{-1}L\left(g_{k}^+\ell_1(n)^\A\right).\]
Moreover, on $C_n^{\delta,K}$, if $nx< \omega_{a_1^+}$, then
\[ n^{-1}L\left(n^{-1}\Delta_{nx}^+\right)\leq n^{-1}L\left(g_{0}^+\ell_1(n)^\A\right).\]
Next, observe that if  $nx\in [\omega_{a_k^+},\omega_{a_{k+1}^+})$ for some $k\geq 1$, then $m_n^+(x)=n^{-1}L(g_k^+)$, and if $nx<\omega_{a_1^+}$, then $m_n^+(x)\leq n^{-1}L(g_0^+)$. Consequently, on $C_n^{\delta,K}$,
\begin{eqnarray*}
\lefteqn{\sup_{x<n^{-1} \omega_{a_K^+}}\left| n^{-1}L\left(n^{-1}\Delta_{nx}^+\right)-m_n^+(x)\right|}\\
&\leq& n^{-1}L\left(g_{0}^+\ell_1(n)^\A\right)+\sup_{k=1,\dots, K}
 \left(n^{-1}L\left(g_{k}^+\ell_1(n)^\A\right)- n^{-1}L\left(g_{k}^+\ell_1(n)^{-\B}\right)\right).
 \end{eqnarray*}
That this estimate holds with high probability as $n\rightarrow\infty$, i.e.\
\[\lim_{n\rightarrow\infty}\mathbb{P}^{\alpha,\lambda/n}\left(C_n^{\delta,K}\right)=1,\]
follows from Proposition~\ref{prop:Awhp} and Theorem~\ref{thm:annealed}.

Recall from \eqref{agconv} that the environments for different values of $n$ are coupled in such a way that $n^{-1}L(g_k^+) \to g_k^{+,\Pois}>0$ as $n\to\infty$. Using the elementary fact that, for any $\theta\in\R$ and deterministic sequence $(x_n)_{n\geq 0}$,
\begin{align*}
\lim_{n\to\infty}n^{-1}L(x_n)= x\in(0,\infty)\qquad\implies \qquad \lim_{n\to\infty}n^{-1}L(x_n\ell_1(n)^\theta)=x,
\end{align*}
it follows that
\[\left(n^{-1}L\left(g_k^+\ell_1(n)^{-\B}\right),n^{-1}L\left(g_k^+\ell_1(n)^\A\right)\right)_{k\geq 1}\rightarrow \left(g_k^{+,\Pois},g_k^{+,\Pois}\right)_{k\geq 1}.\]
Combining these observations, we conclude that, for any $x_0,\varepsilon\in(0,\infty)$,
\begin{eqnarray*}
\lefteqn{\mathbb{P}^{\alpha,\lambda/n}\left(\sup_{x\leq x_0}\left| n^{-1}L\left(n^{-1}\Delta_{nx}^+\right)-m_n^+(x)\right|>\varepsilon\right)}\\
&\leq &  \mathbb{P}^{\alpha,\lambda/n}\left(n^{-1}L\left(g_{0}^+\ell_1(n)^\A\right)+\sup_{k=1,\dots, K}
 \left(n^{-1}L\left(g_{k}^+\ell_1(n)^\A\right)- n^{-1}L\left(g_{k}^+\ell_1(n)^{-\B}\right)\right)>\varepsilon\right)\\
 &&+\mathbb{P}^{\alpha,\lambda/n}\left(\left(C_n^{\delta,K}\right)^c\right)+ \mathbb{P}^{\alpha,\lambda/n}\left(n^{-1} \omega_{a_K^+}\leq x_0\right)\\
 &\rightarrow&\mathbf{P}\left(g_0^{+,\Pois}>\varepsilon\right)+\mathbf{P}\left(a_K^{+,\Pois}\leq x_0\right).
\end{eqnarray*}
Finally, by \eqref{gto0}, the first probability above can be made arbitrarily small by taking $\delta$ small. Also, for fixed $\delta$, we can apply \eqref{adiv} to deduce that $\mathbf{P}(a_K^{+,\Pois}\leq x_0)$ can be made arbitrarily small by taking $K$ large. Hence we conclude that
\[\lim_{n\rightarrow\infty}\mathbb{P}^{\alpha,\lambda/n}\left(\sup_{x\leq x_0}\left| n^{-1}L\left(n^{-1}\Delta_{nx}^+\right)-m_n^+(x)\right|>\varepsilon\right)=0,\]
which implies the result.
\end{proof}

\begin{proof}[Proof of Theorem~\ref{thm:mainnew1}(a)]
From the symmetry of the situation, Proposition~\ref{prop:mainnew1} implies a corresponding version for the exceedance times on the negative axis. Together with Proposition~\ref{prop:mainnew1}, this implies Theorem~\ref{thm:mainnew1}(a).
\end{proof}

\begin{proof}[Proof of Theorem~\ref{thm:mainnew2}(a)] Since $(L(r^{\alpha,0}(\omega_j,\omega_{j+1})))_{j\in\mathbb{Z}}$ are i.i.d.\ random variables satisfying \eqref{ltail} the result is an immediate application of \cite[Proposition 4.20]{Res87} (see also \cite{Lamperti}).
\end{proof}

For the remaining parts of the theorems, we will apply a standard fact from \cite{Whitt} about the continuity of the operation of taking an inverse. Indeed, it is known that if $f_n\rightarrow f$ in the Skorohod $J_1$ topology (or indeed the weaker Skorohod $M_2$ topology), where $f_n$ and $f$ are both unbounded functions in $D([0,\infty),[0,\infty))$, then the right-continuous inverses, defined similarly to \eqref{mnpinv}, satisfy $f_n^{-1}\rightarrow f^{-1}$ with respect to the Skorohod $M_1$ topology whenever $f^{-1}(0)=0$, see \cite[Theorem 13.6.1]{Whitt}. (Note that, by the same result, the operation of taking a right-continuous inverse is also measurable with respect to the relevant topologies.)

\begin{proof}[Proof of Theorem~\ref{thm:mainnew1}(b) and Theorem~\ref{thm:mainnew2}(b)] From Theorem~\ref{thm:mainnew1}(a) and Theorem~\ref{thm:mainnew2}(a), we have that
\begin{eqnarray*}
\lefteqn{\left(n^{-1}L\left(n^{-1}\Delta^-_{nx}\right),n^{-1}L\left(n^{-1}\Delta^+_{nx}\right),m_{n,-}(x),m_{n,+}(x)\right)_{x\geq 0}}\\
&\rightarrow& \left(m_-(x),m_+(x),m_-(x),m_+(x)\right)_{x\geq 0}\hspace{80pt}
\end{eqnarray*}
in distribution with respect to topology on $D([0,\infty),\mathbb{R}^4)$ induced by $d_{J_1}$. Applying the result of \cite{Whitt} that was introduced before the proof, to deduce Theorem~\ref{thm:mainnew1}(b) and Theorem~\ref{thm:mainnew2}(b), it will thus be sufficient to check that, almost-surely, all the functions above are unbounded and that $m_-^{-1}(0)=0=m_+^{-1}(0)$.

Now, since the values of $(m_+(x))_{x\in(0,\infty)}$ and $(m_-(x))_{x\in(0,\infty)}$ correspond to $g_k^{+,\Pois}$ and $g_k^{-,\Pois}$ respectively, they are almost-surely unbounded by \eqref{adiv} and satisfy $m_-^{-1}(0)=0=m_+^{-1}(0)$ by \eqref{gok}. Then, the same hold for $m_{n,+}$ and $m_{n,-}$ by Theorem~\ref{thm:mainnew2}(a) and the Skorohod representation theorem. Finally, we know from Theorem~\ref{thm:mainnew1}(a) that
\[n^{1/2}/L(n^{-1}\Delta^+_n)\rightarrow 0\]
in $\mathbb{P}^{\alpha,\lambda/n}$-probability as $n\rightarrow\infty$. Clearly this implies that $\Delta^+_n\rightarrow \infty$ in $\mathbb{P}^{\alpha,\lambda/n}$-probability. Hence there exists a subsequence $(n_i)_{i\geq 1}$ along which $\Delta^+_{n_i}$ diverges almost-surely (by \cite[Lemma 4.2]{Kall}, for example). Since $(\Delta^+_x)_{x\geq 0}$ is non-decreasing, it follows that $(\Delta_x^+)_{x\geq 0}$ diverges almost-surely. Consequently, for each $n$, the same is true of $(n^{-1}L(n^{-1}\Delta^+_{nx}))_{x\geq 0}$. The same argument gives the corresponding result for $(n^{-1}L(n^{-1}\Delta^-_{nx}))_{x\geq 0}$, and this completes the proof.
\end{proof}

\subsection{Finite dimensional distribution convergence}

We now turn to the proof of Theorem~\ref{thm:mainnew3}, starting with the one-dimensional marginal. Specifically, we will establish the following proposition.

\begin{proposition}\label{onedm}
For any $0\leq s<t$ and continuous bounded function $f$, as $n\rightarrow\infty$,
\[\sup_{\omega_i\in [nm_{n,-}^{-1}(t),nm_{n,+}^{-1}(t)]}\left|{E}^{\alpha,\lambda/n}_\omega\left(f\left(n^{-1}{X}_{nL^{-1}(nt)-nL^{-1}(ns)}\right)\:\vline\:X_0=\omega_i\right)-\frac{\int_{m_{n,-}^{-1}(t)}^{m_{n,+}^{-1}(t)}e^{2\lambda x}f(x)dx}{\int_{m_{n,-}^{-1}(t)}^{m_{n,+}^{-1}(t)}e^{2\lambda x}dx}\right|\]
converges to zero in $\mathbf{P}$-probability.
\end{proposition}

\begin{proof}[Proof of Proposition~\ref{onedm}] We start by considering the case when $s=0$ and $X_0=0$. Fix $\delta,\eta>0$ and $K\in\mathbb{N}$. Recall that in Definition~\ref{def:E}, $E_n^{\delta,\eta,K}$ was defined to be the event that $A_n^{\delta,K}$ holds and, moreover,
\begin{equation}\label{ineq1}
n^{-1}L(\gtt{k}),n^{-1}L(g_{k-1}^+)\leq t-\eta<t+\eta\leq n^{-1}L(g_{k}^+)
\end{equation}
for some $k\in\{1,\dots,K\}$, where $t$ is as in Proposition~\ref{onedm}. The previous inequalities imply that
\begin{equation}\label{ddddd}
m_{n,-}^{-1}(t)=n^{-1}\omega_{\at{k}},\qquad m_{n,+}^{-1}(t)=n^{-1}\omega_{a_k^+}.
\end{equation}
For the time being we will suppose that $E_n^{\delta,\eta,K}$ holds, and write $k$ for the index such that \eqref{ineq1} holds. Recall the definition of the random walk $\tilde{X}^k$ on $\tilde{V}_k=\{\omega_{\at{k}},\dots,\omega_{a_k^+}\}$ from \eqref{xtildekdef}. By construction, $\tilde{X}^k$ behaves the same as $X^k$ up until the time that the latter process exits $\tilde{V}_k$, and thus $\tilde{X}^k$ can be coupled with $X$ so that it follows the same path up until the stopping time $\min\{\betat{k},{\beta}_k^+\}$ (assuming both processes are started from 0). Consequently,
\begin{eqnarray*}
\lefteqn{\left|{E}^{\alpha,\lambda/n}_\omega\left(f\left(n^{-1}{X}_{nL^{-1}(nt)}\right)\right)-E_k\left(f\left(n^{-1}\tilde{X}^k_{nL^{-1}(nt)}\right)\right)\right|}\\
&\leq &\|f\|_\infty{P}^{\alpha,\lambda/n}_\omega\left(\min\{\betat{k},{\beta}_k^+\}\leq nL^{-1}(nt)\right)\\
&\leq &\|f\|_\infty{P}_k\left(S_{G-1}\leq nL^{-1}(nt)\right),
\end{eqnarray*}
where we define $S_{G-1}$ as in Subsection~\ref{sec33}. Noting from \eqref{ineq1} that, for large $n$,
\begin{equation}\label{lntbound}
nL^{-1}(nt)\leq\frac{ng_k^+L^{-1}(nt)}{L^{-1}(n(t+\eta))}=ng_k^+e^{\log^\alpha(nt)-\log^\alpha(n(t+\eta))}\leq ng_k^+e^{-C\log^{\alpha-1}(nt)}< \frac{ng_k^+}{\ell_1(n)^{\B}},
\end{equation}
we find that
\begin{equation}\label{abound}
\left|{E}^{\alpha,\lambda/n}_\omega\left(f\left(n^{-1}{X}_{nL^{-1}(nt)}\right)\right)-E_k\left(f\left(n^{-1}\tilde{X}^k_{nL^{-1}(nt)}\right)\right)\right|\leq
\|f\|_\infty{P}_k\left(S_{G-1}<\frac{ng_k^+}{\ell_1(n)^{\B}}\right).
\end{equation}
Next, we apply the following bound concerning the mixing of the Markov chain $\tilde{X}^k$, which will be proved in Appendix~\ref{app:mixingt}: if $\tilde{\pi}^k$ is the invariant probability measure of $\tilde{X}^k$ and
\begin{align*}
  \tilde{t}^k_{\mathrm{mix}}:=\sup_{\at{k}\leq i,j\leq a_k^+}R_\mathrm{eff}^k(\omega_i,\omega_j) \sum_{i=\at{k}}^{a_k^+}c^{\alpha,\lambda/n}(\omega_i),
\end{align*}
then for any continuous bounded function $f$,
\begin{equation}\label{mixingt}
  \left|E_k\left(f\left(n^{-1}\tilde{X}^k_{nL^{-1}(nt)}\right)\right)-\int f(n^{-1}x)\tilde{\pi}^k(dx)\right|\leq \|f\|_\infty e^{-nL^{-1}(nt)/\tilde{t}^k_{\mathrm{mix}}}\sup_{i=\at{k},...,a_k^+} \tilde \pi^k(\omega_i)^{-1/2}.
  \end{equation}
Note that on the set $\tilde{V}_k$, $\tilde{\pi}^k(\{\omega_i\})\propto c^{\alpha,\lambda/n}_k(\omega_i)= c^{\alpha,\lambda/n}(\omega_i)$, and so
\begin{equation}\label{measureeq}
\int f(n^{-1}x)\tilde{\pi}^k(dx)=\frac{\sum_{i=\at{k}}^{a_k^+}f(n^{-1}\omega_i)c^{\alpha,\lambda/n}(\omega_i)}{\sum_{i=\at{k}}^{a_k^+}c^{\alpha,\lambda/n}(\omega_i)}=\frac{    \int_{m_{n,-}^{-1}(t)}^{m_{n,+}^{-1}(t)}f(x)\mu_n^{\alpha,\lambda,\omega}(\dd x)}{    \int_{m_{n,-}^{-1}(t)}^{m_{n,+}^{-1}(t)}\mu_n^{\alpha,\lambda,\omega}(\dd x)},
\end{equation}
where $\mu_n^{\alpha,\lambda,\omega}$ was defined at \eqref{mundef}. On $E_n^{\delta,\eta,K}$, due to \eqref{eq:pointwise} and \eqref{eq:mass_control}, it holds that
\begin{align*}
\sup_{i=\at{k},...,a_k^+} \frac{1}{\tilde \pi^k(\omega_i)}=\sum_{i=\at{k}}^{a_k^+}c^{\alpha,\lambda/n}(\omega_i)\sup_{i=\at{k},...,a_k^+}\frac{1}{c^{\alpha,\lambda/n}(\omega_i)}\leq c n\ell_1(n)^2 e^{\log^{\alpha+1}n}.
\end{align*}
Furthermore, on $E_n^{\delta,\eta,K}$, we have the bound
\[\tilde{t}^k_{\mathrm{mix}}= \sup_{\at{k}\leq i,j\leq a_k^+}R_\mathrm{eff}^k(\omega_i,\omega_j) \sum_{i=\at{k}}^{a_k^+}c^{\alpha,\lambda/n}(\omega_i)\leq  C n(g_{k-1}^++\gtt{k})\ell_1(n)^3,\]
where we have applied \eqref{eq:mass_control} and Lemma~\ref{lem:R_bounds} to deduce the inequality. Therefore, again recalling \eqref{ineq1}, for large $n$,
\begin{equation}\label{ggg}
\frac{nL^{-1}(nt)}{\tilde{t}^k_{\mathrm{mix}}}\geq \frac{L^{-1}(nt)}{C(g_{k-1}^++\gtt{k})\ell_1(n)^3}\geq
\frac{L^{-1}(nt)}{CL^{-1}(n(t-\eta))\ell_1(n)^3}\geq e^{c\log^{\alpha-1} (nt)},
\end{equation}
and combining this bound with \eqref{abound}, \eqref{mixingt} and \eqref{measureeq} yields that
\begin{align*}
\lefteqn{\left|{E}^{\alpha,\lambda/n}_\omega\left(f\left(n^{-1}{X}_{nL^{-1}(nt)}\right)\right)-\frac{    \int_{m_{n,-}^{-1}(t)}^{m_{n,+}^{-1}(t)}f(x)\mu_n^{\alpha,\lambda,\omega}(\dd x)}{    \int_{m_{n,-}^{-1}(t)}^{m_{n,+}^{-1}(t)}\mu_n^{\alpha,\lambda,\omega}(\dd x)}\right|\times\mathbf{1}_{E_n^{\delta,\eta,K}}(\omega)}\\
&\leq\|f\|_\infty\left({P}_k\left(S_{G-1}<\frac{ng_k^+}{\ell_1(n)^{\B}}\right)\times\mathbf{1}_{E_n^{\delta,\eta,K}}(\omega)+ c n\ell_1(n)^2 e^{\log^{\alpha+1}n-Ce^{c\log^{\alpha-1} (nt)}}\right),
\end{align*}
where, at the expense of including the indicator function $\mathbf{1}_{E_n^{\delta,\eta,K}}$, we have dropped our assumption that $E_n^{\delta,\eta,K}$ holds. Applying Proposition~\ref{prop:inv_meas} and that
\[{P}_k\left(S_{G-1}<\frac{ng_k^+}{\ell_1(n)^{\B}}\right)\times\mathbf{1}_{E_n^{\delta,\eta,K}}\rightarrow 0\]
in $\mathbf{P}$-probability, which follows from~\eqref{eq:SonA}, we can thus conclude that
\[\left|{E}^{\alpha,\lambda/n}_\omega\left(f\left(n^{-1}{X}_{nL^{-1}(nt)}\right)\right)-
\frac{\int_{m_{n,-}^{-1}(t)}^{m_{n,+}^{-1}(t)}e^{2\lambda x}f(x)dx}{\int_{m_{n,-}^{-1}(t)}^{m_{n,+}^{-1}(t)}e^{2\lambda x}dx}\right|\times\mathbf{1}_{E_n^{\delta,\eta,K}}(\omega)\rightarrow0\]
in $\mathbf{P}$-probability.

Our subsequent step is to describe how to extend the result of the previous paragraph to hold for $0\leq s<t$ and arbitrary starting points in $[nm_{n,-}^{-1}(t),nm_{n,+}^{-1}(t)]$. Firstly, for such $s$, we have that: for any $\eta>0$, for large $n$,
\[L^{-1}\left(n(t-\eta/2)\right)\leq L^{-1}(nt)-L^{-1}(ns)\leq L^{-1}(nt).\]
Hence, arguing similarly to \eqref{ggg}, on $E_n^{\delta,\eta,K}$, for $n$ large, we have that
\[\frac{nL^{-1}(nt)-nL^{-1}(ns)}{\tilde{t}^k_{\mathrm{mix}}}\geq  e^{c\log^{\alpha-1} (n(t-\eta/2))}.\]
Thus, proceeding as for the case when $X_0=0$, we find that, uniformly over starting points $\omega_i\in [nm_{n,-}^{-1}(t),nm_{n,+}^{-1}(t)]$,
\begin{eqnarray*}
\lefteqn{\left|{E}^{\alpha,\lambda/n}_\omega\left(f\left(n^{-1}{X}_{nL^{-1}(nt)-nL^{-1}(ns)}\right)\:\vline\:X_0=\omega_i\right)-\frac{\int_{m_{n,-}^{-1}(t)}^{m_{n,+}^{-1}(t)}e^{2\lambda x}f(x)dx}{\int_{m_{n,-}^{-1}(t)}^{m_{n,+}^{-1}(t)}e^{2\lambda x}dx}\right|\times\mathbf{1}_{E_n^{\delta,\eta,K}}(\omega)}\\
&\leq &\|f\|_\infty\sup_{\omega_j\in [nm_{n,-}^{-1}(t),nm_{n,+}^{-1}(t)]}{P}_k\left(S_{G-1}<\frac{ng_k^+}{\ell_1(n)^{\B}}\:\vline\:X_0=\omega_j\right)\times\mathbf{1}_{E_n^{\delta,\eta,K}}(\omega)\\
&&\qquad\qquad\qquad\qquad\qquad\qquad\qquad\qquad\qquad\qquad+\|f\|_\infty c n\ell_1(n)^2 e^{\log^{\alpha+1}n-Ce^{c\log^{\alpha-1} (nt)}}.
\end{eqnarray*}
Since $\omega\cap[nm_{n,-}^{-1}(t),nm_{n,+}^{-1}(t)]=\tilde{V}_k$, the probability involving $S_{G-1}$ can be bounded in almost the same way as \eqref{eq:SonA}. The only adaptation needed in the proof concerns the estimate for $\tau(0)$, as presented at \eqref{taubound}. Specifically, on $E_n^{\delta,\eta,K}$, similarly to the steps leading to \eqref{taubound}, using Lemma~\ref{lem:R_bounds} we can check that:
\begin{align*}
\sup_{\omega_j\in [nm_{n,-}^{-1}(t),nm_{n,+}^{-1}(t)]}\Pk\left(\tau(0)=\infty\:\vline\:X_0=\omega_j\right)&\leq
\sup_{\omega_j\in\tilde{V}_k}\frac{{R}^k_\mathrm{eff}(\omega_j,\omega_{a_k^+})}{{R}^k_\mathrm{eff}(\omega_j,\{\omega_{\bt{k}},\omega_{{b}_k^+}\})}\\
&\leq \frac{C\ell_1(n)^6\left(g_{k-1}^++\gtt{k}\right)}{g_k^+}
\end{align*}
and the right-hand side here can be made arbitrarily small by taking $n$ large. On the other hand, if $\tau(0)<\infty$, then we can use the strong Markov property at $\tau(0)$ and Remark~\ref{rem:SonA}. Thus we can conclude that
\[\sup_{\omega_j\in [nm_{n,-}^{-1}(t),nm_{n,+}^{-1}(t)]}{P}_k\left(S_{G-1}<\frac{ng_k^+}{\ell_1(n)^{\B}}\:\vline\:X_0=\omega_j\right)\times\mathbf{1}_{E_n^{\delta,\eta,K}}(\omega)\xrightarrow{n\to\infty} 0\]
in $\mathbf{P}$-probability. Summarising, this establishes that the supremum over starting points $\omega_i\in [nm_{n,-}^{-1}(t),nm_{n,+}^{-1}(t)]$ of
\begin{equation}\label{summary}
  \left|{E}^{\alpha,\lambda/n}_\omega\left(f\left(n^{-1}{X}_{nL^{-1}(nt)-nL^{-1}(ns)}\right)\:\vline\:X_0=\omega_i\right)-\frac{\int_{m_{n,-}^{-1}(t)}^{m_{n,+}^{-1}(t)}e^{2\lambda x}f(x)dx}{\int_{m_{n,-}^{-1}(t)}^{m_{n,+}^{-1}(t)}e^{2\lambda x}dx}\right|\times\mathbf{1}_{E_n^{\delta,\eta,K}}(\omega)
\end{equation}
converges to zero in $\mathbf{P}$-probability as $n\to\infty$.

Finally, observe that by symmetry, we have that \eqref{summary} still converges to zero in $\mathbf{P}$-probability when ${E}_n^{\delta,\eta,K}$ is replaced by $\tilde{E}_n^{\delta,\eta,K}$. By Proposition~\ref{prop:Ewhp}, we know that $\liminf_{n\rightarrow\infty} \mathbf{P}({E}_n^{\delta,\eta,K}\cup \tilde{E}_n^{\delta,\eta,K})$ can be made arbitrarily close to one by adjusting $\delta$, $K$ and $\eta$. This completes the proof.
\end{proof}

\begin{rem}\label{locremproof} The above proof contains the main ideas needed to check the localisation result of Remark~\ref{locrem}. Indeed, the argument established that
\begin{align*}
\lefteqn{{P}^{\alpha,\lambda/n}_\omega\left(n^{-1}\underline{X}_{nL^{-1}(n t)}<m_{n,-}^{-1}(t)\mbox{ or }n^{-1}\overline{X}_{nL^{-1}(n t)}>m_{n,+}^{-1}(t)\right)\times\mathbf{1}_{E_n^{\delta,\eta,K}}(\omega)}\\
&\leq{P}^{\alpha,\lambda/n}_\omega\left(\min\{\betat{k},{\beta}_k^+\}\leq nL^{-1}(nt)\right)\times\mathbf{1}_{E_n^{\delta,\eta,K}}(\omega),\hspace{120pt}\\
&\rightarrow 0
\end{align*}
in $\mathbf{P}$-probability, where $k$ is the index satisfying \eqref{ineq1}. The symmetry of the situation means we can deduce a similar result on $\tilde{E}_n^{\delta,\eta,K}$, and applying Proposition~\ref{prop:Ewhp} thus yields
\[{P}^{\alpha,\lambda/n}_\omega\left(n^{-1}\underline{X}_{nL^{-1}(n t)}<m_{n,-}^{-1}(t)\mbox{ or }n^{-1}\overline{X}_{nL^{-1}(n t)}>m_{n,+}^{-1}(t)\right)\rightarrow 0\]
in $\mathbf{P}$-probability. We further have that
\begin{align*}
\lefteqn{{P}^{\alpha,\lambda/n}_\omega\left(n^{-1}\underline{X}_{nL^{-1}(n t)}>m_{n,-}^{-1}(t)\mbox{ or }n^{-1}\overline{X}_{nL^{-1}(n t)}<m_{n,+}^{-1}(t)\right)\times\mathbf{1}_{E_n^{\delta,\eta,K}}(\omega)}\\
&\leq{P}^{\alpha,\lambda/n}_\omega\left(\max\{\alphat{k},{\alpha}_k^+\}> nL^{-1}(nt)\right)\times\mathbf{1}_{E_n^{\delta,\eta,K}}(\omega),\hspace{120pt}
\end{align*}
where $\alphat{k}=\inf\{t\geq 0:\:X_t\leq \omega_{\at{k}}\}$. Applying \eqref{ineq1} and Theorem~\ref{thm:annealed}, one can check that
\[{P}^{\alpha,\lambda/n}_\omega\left({\alpha}_k^+> nL^{-1}(nt)\right)\times\mathbf{1}_{E_n^{\delta,\eta,K}}(\omega)\rightarrow 0\]
in $\mathbf{P}$-probability. The same approach gives the corresponding limit with ${\alpha}_k^+$ replaced by $\alphat{k}$, and we can also deal with the situation on $\tilde{E}_n^{\delta,\eta,K}$ in a similar fashion. Hence, again applying Proposition~\ref{prop:Ewhp}, we find that
\[{P}^{\alpha,\lambda/n}_\omega\left(n^{-1}\underline{X}_{nL^{-1}(n t)}>m_{n,-}^{-1}(t)\mbox{ or }n^{-1}\overline{X}_{nL^{-1}(n t)}<m_{n,+}^{-1}(t)\right)\rightarrow 0\]
in $\mathbf{P}$-probability, which completes the proof of the result claimed in Remark~\ref{locrem}.
\end{rem}

\begin{proof}[Proof of Theorem~\ref{thm:mainnew3}(a)] To deduce the finite-dimensional statement of Theorem~\ref{thm:mainnew3}(a) from Proposition~\ref{onedm}, we will repeatedly apply the Markov property. Specifically, for the first step, observe
\begin{align*}
\lefteqn{E^{\alpha,\lambda/n}_\omega\left(\prod_{i=1}^If_i\left(n^{-1}{X}_{nL^{-1}(nt_i)}\right)\right)}\\
&=E^{\alpha,\lambda/n}_\omega\left(\prod_{i=1}^{I-1}f_i\left(n^{-1}{X}_{nL^{-1}(nt_i)}\right)
E^{\alpha,\lambda/n}_\omega\left(f_I\left(n^{-1}{X}_{nL^{-1}(nt_I)}\right)\:\vline\: X_{nL^{-1}(nt_{I-1})}\right)\right).
\end{align*}
Hence, defining
\[\mathcal{I}_{n,i}:=\frac{\int_{m_{n,-}^{-1}(t_i)}^{m_{n,+}^{-1}(t_i)}e^{2\lambda x}f_i(x)dx}{\int_{m_{n,-}^{-1}(t_i)}^{m_{n,+}^{-1}(t_i)}e^{2\lambda x}dx},\]
we have
\begin{align*}
\lefteqn{\left|E^{\alpha,\lambda/n}_\omega\left(\prod_{i=1}^If_i\left(n^{-1}{X}_{nL^{-1}(nt_i)}\right)\right)-E^{\alpha,\lambda/n}_\omega\left(\prod_{i=1}^{I-1}f_i\left(n^{-1}{X}_{nL^{-1}(nt_i)}\right)\right)
\mathcal{I}_{n,I}\right|}\\
&\leq C E^{\alpha,\lambda/n}_\omega\left(\left|E^{\alpha,\lambda/n}_\omega\left(f_I\left(n^{-1}{X}_{nL^{-1}(nt_I)}\right)\:\vline\: X_{nL^{-1}(nt_{I-1})}\right)-\mathcal{I}_{n,I}\right|\right)\\
&\leq C\sup_{\omega_i\in [m_{n,-}^{-1}(t_{I-1}),m_{n,+}^{-1}(t_{I-1})]}\left|E^{\alpha,\lambda/n}_\omega\left(f_I\left(n^{-1}{X}_{nL^{-1}(nt_I)-nL^{-1}(nt_{I-1})}\right)\:\vline\: X_0=\omega_i\right)-\mathcal{I}_{n,I}\right|\\
&\hspace{80pt}+C P^{\alpha,\lambda/n}_\omega\left(X_{nL^{-1}(nt_{I-1})}\not\in [m_{n,-}^{-1}(t_{I-1}),m_{n,+}^{-1}(t_{I-1})]\right).
\end{align*}
Since $[m_{n,-}^{-1}(t_{I-1}),m_{n,+}^{-1}(t_{I-1})]\subseteq [m_{n,-}^{-1}(t_{I}),m_{n,+}^{-1}(t_{I})]$, the first of these terms converges to zero in $\mathbf{P}$-probability by Proposition~\ref{onedm}. To handle the second term, we consider $E_n^{\delta,\eta,K}$ as in the proof of Proposition~\ref{onedm} with $t=t_{I-1}$. Recalling \eqref{ddddd} and the definition of $S_{G-1}$ from Subsection~\ref{sec33}, we have that
\begin{align*}
\lefteqn{P^{\alpha,\lambda/n}_\omega\left(X_{nL^{-1}(nt_{I-1})}\not\in [m_{n,-}^{-1}(t_{I-1}),m_{n,+}^{-1}(t_{I-1})]\right)\times\mathbf{1}_{E_n^{\delta,\eta,K}}(\omega)}\\
&\leq {P}^{\alpha,\lambda/n}_\omega\left(\min\{\betat{k},{\beta}_k^+\}\leq nL^{-1}(nt_{I-1})\right)\times\mathbf{1}_{A_n^{\delta,K}}(\omega)\hspace{100pt}\\
&\leq {P}_k\left(S_{G-1}<\frac{ng_k^+}{\ell_1(n)^{\B}}\right)\times\mathbf{1}_{A_n^{\delta,K}}(\omega),
\end{align*}
where we apply the bound at \eqref{lntbound} to obtain the second inequality (with $k$ being the index satisfying \eqref{ineq1}). Moreover, by \eqref{eq:SonA}, the upper bound here converges to zero in $\mathbf{P}$-probability. A similar limit holds when $E_n^{\delta,\eta,K}$ is replaced by $\tilde{E}_n^{\delta,\eta,K}$ (again with $t=t_{I-1}$). Since $\liminf_{n\rightarrow\infty}\mathbf{P}(E_n^{\delta,\eta,K}\cup\tilde{E}_n^{\delta,\eta,K})$ can be made arbitrarily close to 1 by adjusting $\delta$, $K$ and $\eta$ by Proposition~\ref{prop:Ewhp}, we can conclude that
\[\left|E^{\alpha,\lambda/n}_\omega\left(\prod_{i=1}^If_i\left(n^{-1}{X}_{nL^{-1}(nt_i)}\right)\right)-E^{\alpha,\lambda/n}_\omega\left(\prod_{i=1}^{I-1}f_i\left(n^{-1}{X}_{nL^{-1}(nt_i)}\right)\right)
\mathcal{I}_{n,I}\right|\rightarrow 0\]
in $\mathbf{P}$-probability. Iterating the argument gives the result.
\end{proof}

\begin{proof}[Proof of Theorem~\ref{thm:mainnew3}(b)] Given Theorem~\ref{thm:mainnew3}(a), to prove the result, it will suffice to check that: for any collection of times satisfying $0<t_1<\dots<t_k$ and continuous bounded functions $f_1,\dots,f_k$,
\[\mathbf{E}\left(\prod_{i=1}^k\frac{\int_{m_{n,-}^{-1}(t_i)}^{m_{n,+}^{-1}(t_i)}e^{2\lambda x}f_i(x)dx}{\int_{m_{n,-}^{-1}(t_i)}^{m_{n,+}^{-1}(t_i)}e^{2\lambda x}dx}\right)\rightarrow\mathbf{E}\left(\prod_{i=1}^k\frac{\int_{m_{-}^{-1}(t_i)}^{m_{+}^{-1}(t_i)}e^{2\lambda x}f_i(x)dx}{\int_{m_{-}^{-1}(t_i)}^{m_{+}^{-1}(t_i)}e^{2\lambda x}dx}\right)\]
as $n\rightarrow\infty$. Since the ratios of integrals within the expectations are continuous bounded functions of their limits, this convergence follows from the fact that
\[\left(m_{n,-}^{-1}(t_i),m_{n,+}^{-1}(t_i)\right)_{i=1}^k\rightarrow\left(m_{-}^{-1}(t_i),m_{+}^{-1}(t_i)\right)_{i=1}^k\]
in distribution as $n\rightarrow\infty$. This is in turn a consequence of Theorem~\ref{thm:mainnew2}(b) and the observation that, for every $t\geq 0$, both $m_-^{-1}$ and $m_+^{-1}$ are $\mathbf{P}$-a.s.\ continuous at $t$, which follows from \cite[Proposition 4.9]{Res87}, for example.
\end{proof}

\appendix

\section{Metrics on the space of \cadlag functions}
\label{metricapp}

For the convenience of readers, we recall here some standard notions of distance on the space of \cadlag functions. For a more detailed introduction, see \cite[Chapters 3 and 12]{Whitt}. Write $D([0,\infty),\mathbb{R})$ for the space of \cadlag functions from $[0,\infty)$ into $\mathbb{R}$. For $f,g\in D([0,\infty),\mathbb{R})$, we define
\begin{equation}\label{dudef}
d_U(f,g):=\int_{0}^\infty e^{-T} \min\left\{1, d^T_U(f,g)\right\}dT,
\end{equation}
where
\[d^T_U(f,g):=\sup_{t\in[0,T]}\left|f(t)-g(t)\right|.\]
This is a metric on $D([0,\infty),\mathbb{R})$ that characterises the topology of uniform convergence on compacts. To capture the Skorohod $J_1$ topology, we define $d_{J_1}(f,g)$ similarly to \eqref{dudef}, but with $d^T_U(f,g)$ replaced by
\[d^T_{J_1}(f,g):=\inf_{\lambda}\left(d^T_U(f\circ\lambda,g)+\sup_{t\in[0,T]}\left|\lambda(t)-t\right|\right),\]
where the infimum is over increasing homeomorphisms $\lambda:[0,T]\mapsto[0,T]$. For the Skorohod $M_1$ topology, we need to introduce notation for the completed (partial) graph of $f\in D([0,\infty),\mathbb{R})$. In particular, we set
\[\Gamma^T_f:=\left\{(t,x)\in[0,T]\times \mathbb{R}:\:x\in[\min\{f(t^-),f(t)\},\max\{f(t^-),f(t)\}]\right\}.\]
Moreover, we define a parametric representation of $\Gamma^T_f$ to be a continuous surjection $(u,v):[0,1]\mapsto \Gamma^T_f$ such that, if $s<t$, then either $u(s)<u(t)$ or $u(s)=u(t)$ and $|f(u(s)^-)-v(s)|\leq |f(u(t)^-)-v(t)|$. We then define $d_{M_1}(f,g)$ similarly to \eqref{dudef}, but with $d^T_U(f,g)$ replaced by
\[d^T_{M_1}(f,g):=\inf_{(u_f,v_f),(u_g,v_g)}\left(\sup_{t\in[0,1]}\left|u_f(t)-u_g(t)\right|+\sup_{t\in[0,1]}\left|v_f(t)-v_g(t)\right|\right),\]
where the infimum is over parametric representations $(u_f,v_f)$ of $\Gamma^T_f$ and $(u_g,v_g)$ of $\Gamma^T_g$. Finally, in the course of the article, we sometimes use the same notation to discuss distances between functions in $D([0,\infty),\mathbb{R}^d)$. If $f,g\in D([0,\infty),\mathbb{R}^d)$, then our convention is that
\[d_{U}(f,g):=\sum_{i=1}^dd_{U}(f_i,g_i),\]
where $f_i$ and $g_i$ are the components of $f$ and $g$, respectively, i.e.\ $f(t)=(f_1(t),\dots,f_d(t))$ and $g(t)=(g_1(t),\dots,g_d(t))$. We similarly define $d_{J_1}$ and $d_{M_1}$ on $D([0,\infty),\mathbb{R}^d)$.

\section{Mixing time estimate}\label{app:mixingt}

We prove the bound at \eqref{mixingt} from the proof of Theorem~\ref{thm:mainnew3}. As we shall see below, this is essentially a bound on a mixing time and a similar result is proved in \cite[Lemma 4.1, Corollary 4.2]{NP08} for discrete time random walks. The proof in that paper can be adapted to our continuous time random walks, but we provide a slightly different proof for readers' convenience. Our proof will be given in a general setting since it simplifies the notation and also the argument might be of general interest.

Let $G$ be a finite connected graph (with no loops or multiple edges) and suppose that for each edge $\{x,y\}$ in $G$, a nonnegative weight $c_{x,y}=c_{y,x}> 0$ is assigned. We consider the continuous random walk $((Y_t)_{t\ge 0}, \{P_x\}_{x\in G})$ whose generator $\Delta$ is defined in the same way as~\eqref{generator}. The invariant measure $\mu$ and the effective resistance $R_{\text{eff}}$ are also defined in the same way as~\eqref{eq:def_invariant} and~\eqref{eq:def_Reff}. Let $\pi=\mu(\cdot)/\mu(G)$ denote the invariant distribution. Note first that we have, for any $x\in G$,
\[\left|E_x(f(Y_t))-\int f \dd \pi\right| \le \sum_{y\in G} |f(y)| |P_x(Y(t)=y)-\pi(y)|\le 2\|f\|_\infty d_{\text{TV}}(P_x(Y(t)\in\cdot),\pi),\]
where we write $d_{\text{TV}}$ for the total variation distance between measures on $G$. Thus it suffices to obtain the following bound on the total variation distance
\begin{align}
  \label{eq:TVgoal}
  \sup_{x\in G}d_{\text{TV}}(P_x(Y(t)\in\cdot),\pi)
  \le \frac{1}{2}\exp^{-t/(\mu(G)\text{diam}_{R_{\text{eff}}}(G))}\sup_{x\in G}\pi(x)^{-1/2}.
\end{align}
By Jensen's inequality, we can bound the total variation distance as follows:
\begin{align}
  d_{\text{TV}}(P_x(Y(t)\in\cdot),\pi)
  &=\frac{1}{2}\|\pi(\cdot)^{-1}P_x(Y(t)\in\cdot)-1\|_{L^1(\pi)}\nonumber\\
  &\le \frac{1}{2\mu(G)^{1/2}}\|\pi(\cdot)^{-1}P_x(Y(t)\in\cdot)-1\|_{L^2(\mu)}.\label{eq:L^1-L^2}
\end{align}
Furthermore, we know from~\cite[Corollary 5.9]{Barbook} that
\begin{equation}
  \label{eq:L^2-dist}
  \begin{split}
  \|\pi(\cdot)^{-1}P_x(Y(t)\in\cdot)-1\|_{L^2(\mu)} &\le e^{-t \lambda_2(G)} \|\pi(\cdot)^{-1}\mathbf{1}_{\{x\}}-1\|_{L^2(\mu)}\\
  &= \mu(G)^{1/2}e^{-t\lambda_2(G)}\left(\frac{1}{\pi(x)}-1\right)^{1/2},
  \end{split}
\end{equation}
for any $t\ge 0$, where
\begin{align*}
  \lambda_2(G)=\inf\left\{\mathcal{E}(f,f)\colon \|f\|_{L^2(\mu)}=1 \text{ and }\sum_{x\in G}f(x)\mu(x) =0\right\}
\end{align*}
is the so-called spectral gap. The operator $\mathcal E$ was defined in \eqref{eq:def_E} for the weighted graph corresponding to the Mott random walk, and the definition for arbitrary weighted graphs is analogous. We note that $\lambda_2(G)$ is the second smallest eigenvalue of $-\Delta$ and the above variational expression can be found in~\cite[Lemma 3.36]{Barbook}. 

Now we are going to relate $\lambda_2(G)$ to $\mu(G)\text{diam}_{R_{\text{eff}}}(G)$. To this end we use the variational representation of the effective resistance:
\begin{align}
  \label{eq:Reff-var}
  R_{\text{eff}}(x,y)=\sup\left\{\frac{|f(x)-f(y)|^2}{\mathcal{E}(f,f)}\colon \mathcal{E}(f,f)>0\right\},
\end{align}
which follows readily from \eqref{eq:def_Reff}. Let $\phi$ be an $L^2(\mu)$-normalized eigenfunction of $-\Delta$ associated with $\lambda_2(G)$. Then, since $\mathcal{E}(\phi,\phi)=\lambda_2(G)$, we can use \eqref{eq:Reff-var} to get
\begin{align}
  \label{eq:ReffandLambda}
  R_{\text{eff}}(x,y)\ge \frac{|\phi(x)-\phi(y)|^2}{\lambda_2(G)}.
\end{align}
From $\mu(G)\max_{x\in G}|\phi(x)|^2\ge\|\phi\|_{L^2(\mu)}^2=1$, it follows that there exists $x^*\in G$ such that $|\phi(x^*)|\ge \mu(G)^{-1/2}$. On the other hand, the condition $\sum_{x\in G} \phi(x)\mu(x)=0$ for the eigenfunction $\phi$ implies that there exists $y^*\in G$ with $\phi(x^*)\phi(y^*)<0$. Then we have $|\phi(x^*)-\phi(y^*)|\ge \mu(G)^{-1/2}$. Substituting this into \eqref{eq:ReffandLambda} and rearranging, we get
\begin{align*}
   \lambda_2(G)\ge \frac{1}{\mu(G) R_{\text{eff}}(x^*,y^*)} \ge \frac{1}{\mu(G)\text{diam}_{R_{\text{eff}}}(G)}.
\end{align*}
Combining this bound with~\eqref{eq:L^1-L^2} and~\eqref{eq:L^2-dist}, we obtain \eqref{eq:TVgoal}.

\section*{Acknowledgments}
This research was supported by JSPS Grant-in-Aid for Scientific Research (A) 17H01093, JSPS Grant-in-Aid for Scientific Research (C) 19K03540, JSPS Grant-in-Aid for Scientific Research (C) 21K03286, Grant-in-Aid for JSPS Fellows 19F19814, and the Research Institute for Mathematical Sciences, an International Joint Usage/Research Center located in Kyoto University.

\bibliographystyle{amsplain}
\bibliography{next_mott}

\end{document}